\documentclass[11pt,centertags,oneside]{amsart}

\usepackage{amsmath,amstext,amsthm,amscd,typearea,hyperref}
\usepackage{amssymb}
\usepackage{a4wide}
\usepackage[mathscr]{eucal}
\usepackage{mathrsfs}
\usepackage{charter}
\usepackage{xcolor}
\usepackage{tikz}
\usetikzlibrary{positioning}

\usepackage[a4paper,width=16.5cm,top=3cm,bottom=3cm]{geometry}

\numberwithin{equation}{section}

\allowdisplaybreaks
\newtheorem{\thethmA}{\Alph{thmA}}

\newtheorem{theorem}{Theorem}[section]
\newtheorem{definition}[theorem]{Definition}
\newtheorem{proposition}[theorem]{Proposition}
\newtheorem{corollary}[theorem]{Corollary}
\newtheorem{lemma}[theorem]{Lemma}
\newtheorem{remark}[theorem]{Remark}

\newcommand{\NN}{{\mathbb{N}}}
\newcommand{\ZZ}{{\mathbb{Z}}}

\newcommand{\CC}{{\mathbb{C}}}

\newcommand{\RR}{{\mathbb{R}}}

\newcommand{\QQ}{{\mathbb{Q}}}

\newcommand{\GL}{\mathrm{GL}}
\newcommand{\SL}{\mathrm{SL}}
\newcommand{\Sp}{\mathrm{Sp}}

\newcommand{\SU}{\mathrm{{SU}}}
\newcommand{\SO}{{\mathrm{SO}}}

\newcommand{\Or}{\mathrm{O}}

\newcommand{\Ad}{{\mathrm{Ad}}}

\newcommand{\diag}{\mathrm{diag}}
\newcommand{\rank}{\mathrm{rank}}

\newcommand{\tr}{\mathrm{tr}}

\newcommand{\Sym}{\mathrm{Sym}}
\newcommand{\trans}{^{\mathrm{T}}}
\newcommand{\norm}[1]{\left\lVert #1\right\rVert}
\newcommand{\infnorm}[1]{\left\lVert #1\right\rVert_\infty}
\newcommand{\opinfnorm}[1]{\left\lVert #1 \right\rVert_{\mathrm{op}, \infty}}
\newcommand{\dd}{\,\mathrm{d}}

\providecommand{\NN}{\mathbb{N}}
 \providecommand{\QQ}{Q}

\newcommand{\opnorm}[1]{\left \lVert #1 \right \rVert_{\mathrm{op}}}

\newcommand{\rem}[1]{ }

\usepackage{fancyhdr}

\title{\textbf{Diophantine approximation by Lagrangian frames in symplectic lattices}}

\author{Minchang Kim}
\address{Department of Mathematical Sciences, Seoul National University}
\email{chang011@snu.ac.kr}

\subjclass[2020]{Primary 11J83; Secondary 37A17, 11H06}

\date{}

\begin{document}

\begin{abstract}
Classical Diophantine approximation searches for a short integral vector. In higher rank, one may retain either the rational subspace spanned by several such vectors or the integral frame itself. We study the latter problem for primitive Lagrangian frames $(Q,P) \in \mathrm{Mat}_{d \times 2d}(\mathbb{Z})$ and symmetric matrices. We first prove a Dirichlet theorem showing that, for symmetric targets, primitive Lagrangian frames attain the classical approximation rate: for every $X\in\mathrm{Sym}_d(\mathbb{R})$ and $N\geq1$, there exists such a frame satisfying
\[
\lVert QX+P\rVert_{\mathrm{op}}\leq \frac{c_d}{N},
\qquad
\lVert Q\rVert_{\mathrm{op}}\leq N.
\]
This contrasts with unrestricted full-frame approximation, where the optimal universal exponent is $1/d$. We also prove a Khintchine-type zero-one law with critical sum
\[
\sum_{q\geq1}q^{\varsigma-1}\psi(q)^{\varsigma},
\qquad
\varsigma=\frac{d(d+1)}2.
\]
Retaining the frame, rather than only its span, creates an infinite change-of-basis multiplicity: the naive frame-counting Siegel transform has infinite second moment for $d\geq3$. We construct a conical regularization of the frame space which reduces this multiplicity to a finite group while retaining a positive proportion at every scale. Combined with span separation, this gives a finite-overlap $L^2$ estimate and the required cusp-volume bounds. The regularization is formulated independently of the symplectic condition.
\end{abstract}

\maketitle

\section{Introduction}\label{sec:intro}

\subsection{Higher-Rank Approximation and the Main Results}

The classical Dirichlet theorem for systems of linear forms states that for every
$X\in\mathrm{Mat}_{n\times m}(\RR)$ and every $N\geq1$, there exists
$(\mathbf q,\mathbf p)\in\ZZ^n\times\ZZ^m\setminus\{(0,0)\}$ such that
\[
    \infnorm{\mathbf qX+\mathbf p}\leq N^{-n/m},
    \qquad
    \infnorm{\mathbf q}\leq N.
\]
From the geometry-of-numbers viewpoint, this is a short-vector theorem for a
suitable unimodular lattice. A natural higher-rank question is what remains
true when one asks for an entire collection of linearly independent
approximating vectors. Note that the square matrix case $n = m$ has the exponent $1$.

There are two natural levels at which one may pass from a single vector to higher rank. One may retain only the rational subspace generated by the
approximating vectors, or one may retain the frame itself and require all
vectors in a chosen integral basis to be simultaneously short. These two
problems retain different geometric information. A subspace height is
naturally controlled by a covolume, and hence, by a product of successive
minima, whereas the existence of a uniformly short basis is governed by the
largest successive minimum. Passing to the Grassmannian therefore forgets precisely the basis data relevant to full-frame approximation.

The vector-to-frame distinction already appears at the Dirichlet level. For unrestricted
square matrices, Ten Have and Tijdeman~\cite{THT96} proved that for every
$X\in\mathrm{Mat}_d(\RR)$ and $N>1$, there exist
\[
    Q\in\mathrm{Mat}_d(\ZZ)\cap\GL_d(\RR),
    \qquad
    P\in\mathrm{Mat}_d(\ZZ),
\]
such that
\[
    \infnorm{QX+P}\leq \frac{1}{N^{1/d}},
    \qquad
    \infnorm Q\leq N,
\]
and that the exponent $1/d$ is optimal in this generality. Thus requiring a
full approximating frame, rather than a single approximating vector, produces
a genuine loss in the universal approximation rate.

The starting point of this paper is that additional geometric compatibility
can fundamentally alter this higher-rank problem. We study the natural
symplectic model, in which the approximating frame is required to be
Lagrangian. Set
\[
\mathcal{V}
=
\left\{
(Q,P)\in\mathrm{Mat}_{d\times2d}(\RR):
\rank(Q,P)=d,\;
QP\trans=PQ\trans
\right\}.
\]
The rows of $(Q,P)\in\mathcal{V}$ form a basis of a Lagrangian subspace of
$\RR^{2d}$. We write $\mathcal{V}_{\mathrm{prim}}$ for the integral
Lagrangian frames whose rows generate a primitive rank-$d$ sublattice of
$\ZZ^{2d}$.

The parameter space of the approximation problem is intrinsic to this
geometry. The open Bruhat cell in the Lagrangian Grassmannian transverse to a
fixed coordinate Lagrangian is naturally identified with $\Sym_d(\RR)$: the
graph of a linear map $X:\RR^d\to\RR^d$ is Lagrangian precisely when
$X=X\trans$. This leads naturally to the approximation of symmetric matrices
by primitive integral Lagrangian frames.

Our first result shows that, on the symmetric locus, primitive Lagrangian frames attain the classical Dirichlet exponent $1$.

\begin{theorem}[Dirichlet theorem for Lagrangian frames]\label{thm:diri}
For $d \geq 2$, there exists a constant $c_d>0$, depending only on $d$, such that for every
$X\in\Sym_d(\RR)$ and every $N\geq1$, there exists
$(Q,P)\in\mathcal{V}_{\mathrm{prim}}$ satisfying
\[
    \opnorm{QX+P}\leq\frac{c_d}{N},
    \qquad
    \opnorm Q\leq N.
\]
Moreover, one can choose $c_d$ with
\[
    \log c_d \leq \frac{1}{2}\log(\frac{5}{4})d^2+O(d\log d).
\]
Consequently, there are infinitely many
$(Q,P)\in\mathcal{V}_{\mathrm{prim}}$ with $Q \neq 0$ satisfying
\[
    \opnorm{QX+P}
    \leq
    \frac{c_d}{\opnorm Q}.
\]
\end{theorem}

The comparison is between two different approximation models. The theorem of Ten Have and Tijdeman assumes that $Q$ is nonsingular,
whereas no such condition is imposed in Theorem~\ref{thm:diri}. The appendix treats instead the concatenated full-frame condition
$\rank(Q,P)=d$. Thus, for arbitrary square matrices, no exponent larger than $1/d$ is universally available under the sole full-frame
nondegeneracy condition, whereas on the symmetric locus the symplectic compatibility with Lagrangian frames yields exponent $1$.

The theorem may be viewed as a frame-level refinement of rational
subspace approximation. A primitive Lagrangian frame $(Q,P)\in \mathcal{V}_{\mathrm{prim}}$ determines a rational Lagrangian subspace
\[
L_{Q,P}:=\operatorname{rowspan}(Q,P) \subset \RR^{2d}.
\]
The inequalities in Theorem~\ref{thm:diri} do not merely produce such a subspace close to the relevant target: they produce a primitive integral basis of $L_{Q,P} \cap \ZZ^{2d}$ whose vectors satisfy simultaneous bounds.

This is stronger than the information encoded by the usual
Grassmannian height. Let
\[
\lambda_1(L)\leq\cdots\leq\lambda_d(L)
\]
be the successive minima of the induced lattice
$L\cap\ZZ^{2d}$. By Minkowski's second theorem, the covolume of this
lattice, and hence its usual covolume--Pl\"ucker height up to
dimension-dependent constants, is comparable to
\[
\lambda_1(L)\cdots\lambda_d(L).
\]
On the other hand, the existence of a primitive integral basis of
$L\cap\ZZ^{2d}$ all of whose vectors have norm at most $R$ is,
up to constants depending only on $d$, equivalent to the bound
\[
\lambda_d(L)\ll_d R.
\]
Thus, covolume height controls the product of the successive minima,
whereas the simultaneous short-frame condition requires control of the
largest one. The latter is a condition on the shape of the induced
lattice and is not determined by its covolume.
Intrinsic Grassmannian approximation, as in the work of Schmidt~\cite{Sch67} and de Saxc\'e~\cite{deS26}, is formulated in
terms of a covolume-Pl\"ucker height and an angular approximation condition. Neither datum alone controls the largest successive minimum
of $L\cap\ZZ^{2d}$, and hence it does not directly yield the simultaneous frame bounds required here.
\[
\text{Vector}
\;\longrightarrow\;
\text{Rational subspace}
\;\xrightarrow{\;\text{choice of an integral basis}\;}
\text{Frame}.
\]
A frame records a choice of integral basis, while the existential short-basis condition considered here is governed by the shape of the induced primitive lattice rather than by covolume height alone.

The classical metric counterpart of Dirichlet's theorem is the
Khintchine--Groshev theorem. Under the usual monotonicity assumption, for
$X\in\mathrm{Mat}_{n\times m}(\RR)$ the set of matrices for which
\[
    \infnorm{\mathbf qX+\mathbf p}
    \leq
    \psi(\infnorm{\mathbf q})
\]
has infinitely many solutions
$(\mathbf q,\mathbf p)\in\ZZ^n\times\ZZ^m$ has full or zero Lebesgue measure
according as
\[
    \sum_{q=1}^{\infty}q^{n-1}\psi(q)^m
\]
diverges or converges; see Khintchine~\cite{Khi24} and
Groshev~\cite{Gro38}. We obtain a frame analogue of this dichotomy for
symmetric matrices. Set
\[
\varsigma=\dim\Sym_d(\RR)=\frac{d(d+1)}2.
\]

\begin{theorem}[Khintchine-type zero-one law]\label{thm:sk}
Let $\psi:[1,\infty)\to(0,\infty)$ be continuous and non-increasing with
$\psi(q)\to0$. Define
\[
W(\psi)
=
\left\{
X\in\Sym_d(\RR)\cap[0,1)^{d\times d}:
\begin{aligned}
&\infnorm{QX+P}\leq\psi(\infnorm Q)\\
&\text{for infinitely many }
(Q,P)\in\mathcal{V}_{\mathrm{prim}},
\quad Q\neq0
\end{aligned}
\right\}.
\]
Then
\[
\mathrm{Leb}(W(\psi))
=
\begin{cases}
0,
&
\displaystyle
\sum_{q=1}^{\infty}
q^{\varsigma-1}\psi(q)^{\varsigma}<\infty,
\\[1.2ex]
1,
&
\displaystyle
\sum_{q=1}^{\infty}
q^{\varsigma-1}\psi(q)^{\varsigma}=\infty.
\end{cases}
\]
Here Lebesgue measure is normalized so that
$\Sym_d(\RR)/\Sym_d(\ZZ)$ has total measure $1$.
\end{theorem}

The critical exponent in Theorem~\ref{thm:sk} is the one suggested by the
$\varsigma$-dimensional parameter space. The main difficulty lies instead in
the counting problem behind the divergence statement. Retaining an integral
basis introduces an infinite change-of-basis multiplicity which disappears
when one passes to the corresponding rational subspace. In fact, the natural
second moment of the frame-counting function diverges for $d\geq3$.
Controlling this basis multiplicity is the main additional ingredient in the
metric theorem.

\subsection{A Frame-Counting Mechanism}

The distinction between subspaces and frames has a simple homogeneous
interpretation. Let
\[
    G=\Sp_{2d}(\RR),
    \qquad
    \Gamma=\Sp_{2d}(\ZZ),
    \qquad
    \mathcal X=\Gamma\backslash G,
\]
and let $P<G$ denote the Siegel parabolic subgroup with unipotent radical
$N$. With our conventions,
\[
    P\backslash G
    \quad\text{parametrizes Lagrangian subspaces,}
    \qquad
    N\backslash G
    \quad\text{parametrizes Lagrangian frames.}
\]
The projection
\[
    N\backslash G\longrightarrow P\backslash G
\]
forgets the basis, and its fiber is
\[
    N\backslash P\simeq\GL_d(\RR).
\]
At the arithmetic level, the corresponding projection
\[
    (\Gamma \cap N)\backslash\Gamma
    \longrightarrow
    (\Gamma \cap P)\backslash\Gamma
\]
retains the $\GL_d(\ZZ)$-family of integral bases of a fixed primitive
Lagrangian lattice.

This residual change-of-basis action is the basic difference between
Lagrangian subspace counting and Lagrangian frame counting. The corresponding subspace count is indexed by $(\Gamma \cap P) \backslash \Gamma$, whereas
the frame-counting Siegel transform is naturally indexed by
$(\Gamma \cap N)  \backslash \Gamma$. Thus even after fixing the underlying
rational Lagrangian subspace, one must still control infinitely many integral
changes of basis.

The first simplification is a separation phenomenon for short frames. If two
primitive Lagrangian frames $v$ and $w$ in the same symplectic lattice span
distinct Lagrangian subspaces, then some pair of their rows has a nonzero
integral symplectic pairing. Hence
\[
    1
    \leq
    |\omega(v_i,w_j)|
    \ll
    \norm{v_i}\norm{w_j},
\]
so two sufficiently short primitive Lagrangian frames must have the same real
span. We refer to this property as \emph{span separation}.

It follows that for a sufficiently small frame region
$E\subset\mathcal{V}$, the pair correlation reduces to the residual
change-of-basis action inside one Lagrangian span:
\[
    \int_{\mathcal X}\widehat{\chi}_E^{\,2}\,\dd m_{\mathcal X}
    =
    \sum_{\delta\in\GL_d(\ZZ)}
    m_{\mathcal{V}}(E\cap\delta^{-1}E).
\]
However, the sum on the right is
infinite when $d\geq3$; see
Proposition~\ref{prop:divergence_secondmoment}. The divergence is caused by
highly distorted integral bases of one and the same primitive Lagrangian
lattice. This is in sharp contrast with ordinary lattice-point counting,
where the classical second-moment theory for the Siegel transform, beginning
with Rogers~\cite{Rog55}, provides the required integrability. Span
separation therefore removes the multiplicity of distinct subspaces, but not
the basis multiplicity remaining over a fixed subspace.

To control this second source of multiplicity, we introduce a
frame-theoretic regularization which is independent of the symplectic form.
For a rank-$d$ frame
$v\in\mathrm{Fr}_{d,n}\subset\mathrm{Mat}_{d \times n}(\RR)$, call $v$
\emph{rigid} if
\[
    \norm{\delta v}_F\geq\norm v_F
    \qquad
    \text{for every }\delta\in\GL_d(\ZZ),
\]
with equality only for $\delta\in\Or_d(\ZZ)$. The reduction argument used
below applies to any nonempty $\GL_d(\RR)$-invariant frame space
$\mathscr F\subset\mathrm{Fr}_{d,n}$ and produces a nonempty open conical
sector $\mathcal D$ for which
\[
    \mathcal D\cap\delta\mathcal D\neq\varnothing
    \qquad\Longrightarrow\qquad
    \delta\in\Or_d(\ZZ).
\]
Since the sector is stable under scaling, it occupies a fixed positive
proportion at every sufficiently small scale. Restricting the counting
function to this sector therefore replaces the infinite
$\GL_d(\ZZ)$-multiplicity by the finite signed permutation group without
changing the order of the first moment.

The metric argument can thus be separated into two inputs:
\[
\begin{tikzpicture}[
    box/.style={
        draw,
        rounded corners=2pt,
        minimum height=8mm,
        minimum width=38mm,
        inner xsep=6pt,
        inner ysep=3pt,
        align=center
    },
    smallbox/.style={
        draw,
        rounded corners=2pt,
        minimum height=8mm,
        minimum width=46mm,
        inner xsep=6pt,
        inner ysep=3pt,
        align=center
    },
    >=stealth
]

\node[box] (frames) at (0,0)
    {Arithmetic frame system};

\node[smallbox] (span) at (-3.2,-1.8)
    {Span separation\\
    {\small Subspace multiplicity}};

\node[smallbox] (rigid) at (3.2,-1.8)
    {Rigid frame regularization\\
    {\small $\GL_d(\ZZ)$-basis multiplicity}};

\node[box] (moment) at (0,-3.6)
    {Controlled second moment};

\draw[->] (frames) -- (span);
\draw[->] (frames) -- (rigid);
\draw[->] (span) -- (moment);
\draw[->] (rigid) -- (moment);

\end{tikzpicture}
\]
The two inputs have different degrees of generality. Rigid frame
regularization is purely frame-theoretic and is formulated on arbitrary
$\GL_d(\RR)$-invariant frame spaces. Span separation depends on the ambient
arithmetic geometry. In the symplectic case it follows from the integral
symplectic pairing, but the same separation mechanism occurs for maximal
isotropic frame systems associated with other classical groups, including
$\SO(d,d)$ and $\SU(d,d)$; see Remark~\ref{rmk:ss_examples}. Thus the
second-moment argument naturally applies to arithmetic frame systems for
which these two ingredients are available. In the present paper, they are
established for primitive Lagrangian frames and provide the counting input
for the proof of Theorem~\ref{thm:sk}.

\subsection{Proof Strategy and Organization}

The two main results are proved using complementary aspects of reduction
theory and homogeneous dynamics. The proof of
Theorem~\ref{thm:diri} is reduction-theoretic. We apply Siegel reduction on
the Siegel upper half-space $\mathfrak H_d$~\cite{Sie43} to
\[
    Z=X+ic_dN^{-2}I_d.
\]
After reducing $Z$ by an element of $\Sp_{2d}(\ZZ)$, the last $d$ rows of the
reducing element give a primitive integral Lagrangian frame $(Q,P)$, and the
quantitative geometry of the reduced domain gives the required bounds for
$\opnorm Q$ and $\opnorm{QX+P}$.

The proof of Theorem~\ref{thm:sk} follows the dynamical approach to
Diophantine approximation initiated by Dani~\cite{Dan85} and developed
systematically in homogeneous dynamics by Kleinbock--Margulis~\cite{KM98};
see also~\cite{KW08}. A Dani-type correspondence converts approximation by
Lagrangian frames into a shrinking-target problem for a diagonal flow on
$\mathcal X=\Sp_{2d}(\ZZ)\backslash\Sp_{2d}(\RR)$. The frame-counting
mechanism above gives the sharp volume estimate for the corresponding cusp
targets. The dynamical Borel--Cantelli theorem of
Kleinbock--Margulis~\cite{KM99} then yields the convergence--divergence
dichotomy in Theorem~\ref{thm:sk}. Thus the dynamical part follows the
classical Khintchine--Groshev architecture, while the new input is the
geometry and counting of short Lagrangian frames.

The paper is organized as follows.
Section~\ref{sec:homo_symp} develops the geometry of Lagrangian frames, proves Theorem~\ref{thm:diri} and span separation, and
establishes the frame-counting mean-value formula.
Section~\ref{sec:dynamical_formulation} gives the Dani correspondence and the
dynamical formulation of the metric problem.
Section~\ref{sec:cusp_estimate} proves the sharp cusp-volume estimate,
including rigid frame regularization and the resulting second-moment bound,
and completes the proof of Theorem~\ref{thm:sk}.
Appendix~\ref{sec:appendix} proves the obstruction to improving the exponent
$1/d$ in unrestricted full-frame approximation and the divergence of the
naive frame-counting second moment for $d\geq3$.

Throughout, $\|\cdot\|_\infty$, $\|\cdot\|_{\mathrm{op}}$, and
$\|\cdot\|_F$ denote the supremum, operator, and Frobenius norms,
respectively. We write $A\ll B$ if $A\leq CB$ for a constant depending only
on $d$, and $A\asymp B$ if $A\ll B\ll A$.

\subsection*{Acknowledgements}

The author would like to thank Seonhee Lim for her guidance and helpful discussions. The author is supported by the National Research Foundation of Korea under project numbers RS-2025-00515082 and RS-2025-02293115.


\section{Lagrangian Frames and Homogeneous Spaces}\label{sec:homo_symp}
We begin with the reduction theory on the Siegel upper half-space and use it to prove the Dirichlet theorem for Lagrangian frames. We then describe the space of Lagrangian frames as a homogeneous space of the symplectic group and identify primitive integral frames with an arithmetic orbit. This homogeneous description provides the invariant measure and the associated Siegel transform needed for the frame-counting problem.

We conclude the section with the span separation property for primitive Lagrangian frames, which isolates the geometric input that reduces the counting of sufficiently short frames to the change-of-basis action within a single Lagrangian subspace. These ingredients form the geometric and counting framework used in the dynamical formulation of the metric problem in Section~\ref{sec:dynamical_formulation}. 

\subsection{Siegel Upper Half-Space}
Let $\omega(v,w) = vJw\trans$ be the standard symplectic form on $\RR^{2d}$, where
\[
J=\begin{pmatrix}0&I_d\\-I_d&0\end{pmatrix}.
\] 
With our row-vector convention, write
\[
G = \Sp_{2d}(\RR) = \{g \in \SL_{2d}(\RR): gJg\trans = J\}.
\]
The Siegel upper half-space $\mathfrak{H}_d$ was introduced by Siegel as a natural higher-rank generalization of the upper half-plane $\mathfrak{H}_1 = \mathbb{H}$ \cite{Sie39}. To prove the Dirichlet-type theorem for symmetric matrices, we use the reduction theory on the Siegel upper half-space. 

\begin{definition}[Siegel upper half-space and symplectic group action]
The \emph{Siegel upper half-space} is the set
\[
\mathfrak{H}_d=\{Z=X+iY\in\Sym_d(\CC): Y \text{ is positive definite}\}.
\]
The action of $G = \Sp_{2d}(\RR)$ on $\mathfrak{H}_d$ is given by
\[
g\cdot Z=(AZ+B)(CZ+D)^{-1}.
\]
\end{definition}

Let $\Gamma = \Sp_{2d}(\ZZ)$ be the Siegel modular group. Since $\Gamma$ acts properly discontinuously on $\mathfrak{H}_d$, it admits a fundamental domain $\mathcal{F}_d \subset \mathfrak{H}_d$. An explicit fundamental domain was constructed by Siegel \cite{Sie43} using Minkowski reduction on the space of positive definite symmetric matrices (see also \cite{Kli90}).

\begin{definition}[Minkowski reduced cone/domain]
A positive definite $d \times d$ symmetric matrix $Y$ is called \emph{Minkowski reduced} if it satisfies:
\begin{itemize}
    \item[M1] $v\trans Y v \geq Y_{kk}$ for $v \in \ZZ^d$ and $\mathrm{gcd}(v_k, v_{k+1}, \dots, v_d) = 1$, $1 \leq k \leq d$.
    \item[M2] $Y_{k, k+1} \geq 0$, $1 \leq k < d$.
\end{itemize}

\end{definition}

Any $d \times d$ Minkowski reduced matrix $Y$ satisfies
\begin{equation}\label{eq:orthdef}
    \det(Y) \leq Y_{11} \dots Y_{dd} \leq C_d \det(Y),
\end{equation}
where $C_d$ denotes the optimal constant in Minkowski's fundamental
inequality, that is, the smallest constant for which the second inequality
holds for every $d \times d$ Minkowski reduced matrix; see
\cite[Satz~2.5]{Fre83} and \cite{Bar78}. The optimal values of $C_d$ are
known in dimensions $d \leq 5$ \cite{Bar78}, 
\[
C_1 = 1, \quad C_2 = \frac{4}{3}, \quad C_3 = 2, \quad C_4 = 4, \quad C_5 = 8.
\] 
In arbitrary dimension, a classical estimate due to van der Waerden gives, for $d > 4$,
\[
    C_d
    \leq
    \gamma_d^d
    \left(\frac{5}{4}\right)^{(d-3)(d-4)/2},
\]
where $\gamma_d$ denotes the $d$-dimensional Hermite constant \cite{vdW56}. In particular, since $\gamma_d = O(d)$,
\begin{equation}\label{eq:C_d}
    \log C_d
    \leq \frac{1}{2}\log(\frac{5}{4})d^2 + O(d\log d).
\end{equation}

The following lemma establishes a lower bound for any Minkowski reduced matrix in terms of its diagonal entries. While the proof is essentially identical to that in \cite{Fre83}, we include it here to explicitly keep track of the constants.

\begin{lemma}[Folgerung 2.6 \cite{Fre83}]\label{lem:alpha}
    Let $Y$ be a Minkowski reduced symmetric matrix. Then there exists a constant $\alpha_d$, depending only on $d$, such that
    \[
    Y \succeq \alpha_d \diag(Y_{11}, \dots, Y_{dd})
    \]
    where $\succeq$ denotes the Loewner order: $A \succeq B$ iff $A - B$ is positive semidefinite.
\end{lemma}

\begin{proof}
If $d = 1$, $\alpha_1 = 1$ is immediate. Assume $d \geq 2$. Let $D = \diag(Y_{11}^{1/2}, \dots, Y_{dd}^{1/2})$ and define the normalized matrix $B = D^{-1}YD^{-1}$. Since $Y$ is positive definite, $B$ remains positive definite. Furthermore, all diagonal entries of $B$ are equal to $1$, which implies $\tr(B) = d$. Let $0 < \lambda_1 \leq \lambda_2 \leq \dots \leq \lambda_d$ denote the eigenvalues of $B$. By \eqref{eq:orthdef}, the determinant of $B$ satisfies
\[
\det(B) = \prod_{i=1}^{d} \lambda_i = \frac{\det(Y)}{Y_{11} \dots Y_{dd}} \geq C_d^{-1}.
\]
Observe that $\sum_{i=2}^{d} \lambda_i = d - \lambda_1 < d$. Applying the arithmetic-geometric mean inequality to the $d-1$ eigenvalues $\lambda_2, \dots, \lambda_d$, we obtain
\[
\prod_{i=2}^{d} \lambda_i \leq \left( \frac{1}{d-1} \sum_{i=2}^{d} \lambda_i \right)^{d-1} < \left(\frac{d}{d-1}\right)^{d-1}.
\]
Combining with the determinant inequality yields
\[
C_d^{-1} \leq \lambda_1 \prod_{i=2}^{d} \lambda_i < \lambda_1 \left(\frac{d}{d-1}\right)^{d-1}.
\]
Consequently, the smallest eigenvalue $\lambda_1$ is strictly bounded below by the constant $\alpha_d := C_d^{-1} \left(\frac{d-1}{d}\right)^{d-1}$. This establishes $B \succeq \alpha_d I_d$, which is equivalent to the desired inequality $Y \succeq \alpha_d \diag(Y_{11}, \dots, Y_{dd}).$
\end{proof}

Now consider the Siegel reduced domain of $\Gamma \backslash \mathfrak{H}_d$.

\begin{theorem}[Siegel reduced domain]\label{thm:sie}
    Let $Z = X + iY \in \mathfrak{H}_d$ lie in the fundamental domain $\mathcal{F}_d$. Then it satisfies:
    \begin{itemize}
        \item[S1] 
        The imaginary part $Y = \mathrm{Im}(Z)$ is Minkowski reduced.
        \item[S2]
        $|X_{ij}| \leq 1/2$ for $1 \leq i, j \leq d$ where $X = (X_{ij})$.
        \item[S3] 
        $\det(\mathrm{Im}(\gamma \cdot Z)) \leq \det(\mathrm{Im}(Z))$ for all $\gamma \in \Gamma$, equivalently, for every $\gamma = \begin{pmatrix}
            A & B \\ C & D
        \end{pmatrix} \in \Gamma$, we have
        \[
        |\det(CZ + D)| \geq 1.
        \]
        This follows from $\mathrm{Im}(\gamma \cdot Z) = ((C\bar{Z} + D)^{-1})\trans \mathrm{Im}(Z)(CZ + D)^{-1}$.

    \end{itemize}
\end{theorem}

\begin{proof}
    See Lemmas 6--9 in \cite{Sie43}.
\end{proof}

The following lemma is the basic property of matrices in $\mathcal{F}_d$.

\begin{lemma}[Hilfssatz 2.11 \cite{Fre83}]\label{lem:hil}
    Let $Z = X + iY$ be a matrix satisfying (S2) and (S3) of Theorem \ref{thm:sie}. Then it satisfies
    \[
    Y_{ii} \geq \frac{\sqrt{3}}{2}, \qquad i = 1, \dots, d.
    \]
\end{lemma}

This yields a uniform lower bound for the eigenvalues of $Y$.

\begin{theorem}[Satz 2.12 \cite{Fre83}]\label{thm:fre}
    Let $Z = X + iY \in \mathcal{F}_d$. Then there exists a constant $c_d > 0$ depending only on $d$, such that 
    \[
    Y \succeq c_d^{-1} I_d
    \] where $\succeq$ denotes the Loewner order: $A \succeq B$ iff $A-B$ is positive semidefinite.
\end{theorem}

\begin{proof}
    Combining Lemma~\ref{lem:alpha} with Lemma~\ref{lem:hil}, we have
    \[
    Y \succeq \alpha_d \diag(Y_{11}, \dots, Y_{dd}) \succeq \frac{\sqrt{3}}{2}\alpha_d I_d.
    \]
    Take the constant $c_d$ to be $\frac{2}{\sqrt{3}}\alpha_d^{-1} = \frac{2}{\sqrt{3}}C_d (\frac{d}{d-1})^{d-1}$.
\end{proof}

Using \eqref{eq:C_d}, we have $\log c_d \leq \frac{1}{2}\log(5/4) d^2 + O(d \log d)$. Before proving Theorem~\ref{thm:diri}, we recall the notion of a Lagrangian frame.

\begin{definition}[Lagrangian frames]\label{def:lagframe}
Define the space of Lagrangian frames by
\[
\mathcal{V}
=
\{(Q,P) \in \mathrm{Mat}_{d\times 2d}(\RR):
\rank(Q,P)=d,\;
QP\trans = PQ\trans \}.
\]
An integral Lagrangian frame $(Q,P) \in \mathcal{V} \cap
\mathrm{Mat}_{d\times 2d}(\ZZ)$ is called \emph{primitive} if its rows span a
primitive rank-$d$ sublattice of $\ZZ^{2d}$. We denote the set of
primitive Lagrangian frames by $\mathcal{V}_{\mathrm{prim}}$.

\end{definition}

\begin{proof}[Proof of Theorem~\ref{thm:diri}]
Throughout, $\|\cdot\|$ denotes the operator norm induced by the standard Hermitian inner product on $\CC^d$, which restricts to the Euclidean operator norm on $\RR^d$. Let $Z = X + ic_d N^{-2}I_d \in \mathfrak{H}_d$. By Theorem~\ref{thm:sie}, there exists
\[
\gamma =\begin{pmatrix} A & B \\ C & D \end{pmatrix} \in \Sp_{2d}(\ZZ)
\]
such that $\gamma \cdot Z \in \mathcal{F}_d$. Write $\gamma \cdot Z = X' + iY'$. The classical transformation law for the imaginary part on the Siegel upper half-space reads
\begin{equation}\label{eq:imag_part}
Y' \;=\; ((C\bar Z + D)^{-1})\trans\, \mathrm{Im}(Z)\, (CZ + D)^{-1}.
\end{equation}
Taking inverses in \eqref{eq:imag_part}, with $\mathrm{Im}(Z) = c_d N^{-2}I_d$ and hence $\mathrm{Im}(Z)^{-1} = c_d^{-1} N^2 I_d$,
\[
(Y')^{-1} \;=\; c_d^{-1} N^2\, (CZ + D)(C\bar{Z} + D)\trans.
\]
Expanding $CZ + D = (CX + D) + ic_d N^{-2}C$ and $(C\bar Z + D)\trans = (CX + D)\trans - ic_d N^{-2} C\trans$ implies
\[
(CZ+D)(C\bar Z + D)\trans
= (CX+D)(CX+D)\trans + c_d^2 N^{-4}\, CC\trans
+ ic_d N^{-2}\big[C(CX+D)\trans - (CX+D)C\trans\big].
\]
The bracketed term equals $CD\trans - DC\trans$, which vanishes by the symplectic relation $CD\trans = DC^{\trans}$. Therefore
\begin{equation}\label{eq:Yprime_inv}
(Y')^{-1} \;=\; c_d^{-1} N^{2}\,(CX+D)(CX+D)\trans + c_d N^{-2}\, CC\trans,
\end{equation}
which is real, symmetric, and positive semidefinite.

By Theorem~\ref{thm:fre}, $Y' \succeq c_d^{-1} I_d$, hence $(Y')^{-1} \preceq c_d I_d$. Applied to any $v \in \RR^d$, this gives
\[
c_d^{-1} N^2\|(CX+D)\trans v\|_2^2 + c_d N^{-2}\|C\trans v\|_2^2 \leq c_d \|v\|_2^2.
\]
Both summands on the left are non-negative, so each is individually bounded by $c_d \|v\|_2^2$. Taking suprema over $v$ (and using $\opnorm{M} = \opnorm{M\trans}$):
\[
\opnorm{CX+D} \leq \frac{c_d}{N}, \qquad
\opnorm{C} \leq N.
\]
Taking $Q = C$ and $P = D$, the symplectic relation
\[
CD\trans = DC\trans
\]
shows that $(Q,P)$ is a Lagrangian frame. Moreover, since
\[
\gamma \in \Sp_{2d}(\ZZ) \subset \SL_{2d}(\ZZ),
\]
the rows of $\gamma$ form a $\ZZ$-basis of $\ZZ^{2d}$. Hence the last $d$ rows, namely the rows of $(C, D)$, generate a primitive rank-$d$ sublattice of $\ZZ^{2d}$. Therefore
\[
(C, D) \in \mathcal{V}_{\mathrm{prim}},
\]
and Theorem~\ref{thm:diri} follows. The infinite-family conclusion is the standard consequence. Applying the
uniform statement at successive values
$N_1 < N_2 < \cdots \to \infty$ produces frames
$(Q_{N_k},P_{N_k}) \in \mathcal{V}_{\mathrm{prim}}$ such that
\[
\opnorm{Q_{N_k}X+P_{N_k}}
\leq c_dN_k^{-1}\longrightarrow 0.
\]
If $X$ admits no exact integral Lagrangian approximation, then only finitely
many repetitions of any fixed frame are possible, and hence infinitely many
of these frames are distinct. If instead $QX + P = 0$ for some
$(Q, P) \in \mathcal{V}_{\mathrm{prim}}$, then $Q \in \GL_d(\RR)$, and
the frames
\[
(\delta Q,\delta P),\qquad \delta\in\GL_d(\mathbb Z),
\]
give infinitely many distinct primitive exact approximations. Consequently, there are infinitely many
$(Q,P) \in \mathcal{V}_{\mathrm{prim}}$ with $Q \neq 0$ such that
\[
    \opnorm{QX+P}
    \leq
    \frac{c_d}{\opnorm Q}.
\]
\end{proof}

\begin{remark}
The same reduction argument also applies when $d=1$ and gives
$c_1=2/\sqrt{3}$. In dimension one, the uniform statement remains valid
for all $X\in\RR$, while the conclusion asserting the existence of infinitely many primitive frames requires $X \notin \QQ$. The reduction-theoretic argument is not designed to optimize the constant $c_d$. Already when $d = 1$, the above proof gives $c_1=2/\sqrt{3}$, whereas the optimal universal constant in the classical uniform Dirichlet theorem is $1$.
\end{remark}

\subsection{The Space of Lagrangian Frames}
Define
\[
\pi: \mathcal{V} \rightarrow \mathrm{Lag}(d, 2d), \;\; (Q, P) \mapsto \mathrm{rowspan}(Q, P),
\]
where $\mathrm{Lag}(d, 2d)$ is the Lagrangian Grassmannian. Thus $\pi$ forgets the basis of a Lagrangian subspace, and its fibers are the left $\GL_d(\RR)$-orbits. Let 
\[
\mathcal{O} = \{\pi(Q, P) \in \mathrm{Lag}(d, 2d): P \in \GL_d(\RR)\}
\]
be an open Bruhat cell. Consider the following subgroups of $G$:
\begin{align*}
&N = \left\{u_X = \begin{pmatrix}
    I_d & X \\
    0 & I_d
\end{pmatrix}: X \in \Sym_d(\RR)\right\} \\
&L = \left\{\iota(A) = \begin{pmatrix}
(A\trans)^{-1} & 0 \\
0 & A
\end{pmatrix} \in G: A \in \GL_d(\RR)\right\} \\
&N^{-} = \left\{v_Y = \begin{pmatrix}
    I_d & 0 \\
    Y & I_d
\end{pmatrix}: Y \in \Sym_d(\RR)\right\}.
\end{align*}
Let $F=(0,I_d)$ be the standard Lagrangian frame, and let $\pi(F)$ denote its row span. Then
\[
\operatorname{Stab}_G(F)=N,
\qquad
\operatorname{Stab}_G(\pi(F))=P=NL.
\]
Since the Levi factor $L$ changes the choice of basis within a fixed Lagrangian subspace, we obtain the homogeneous-space identifications
\begin{align*}
    &P\backslash G \simeq \operatorname{Lag}(d,2d), &&N\backslash G \simeq \mathcal{V}, \\
    &Pg \mapsto \pi(Fg), && Ng \mapsto Fg
\end{align*}
At the arithmetic level, the corresponding orbits
\[
(\Gamma \cap P) \backslash \Gamma
 \qquad \text{and} \qquad
(\Gamma \cap N) \backslash \Gamma
\]
parametrize primitive Lagrangian sublattices of $\ZZ^{2d}$ and primitive Lagrangian frames in $\ZZ^{2d}$, respectively. Indeed, every primitive Lagrangian frame can be completed to an integral symplectic basis, and hence is of the form $F \cdot \gamma$ for some $\gamma \in \Gamma$. Therefore, $\Gamma$ acts transitively on the set of primitive Lagrangian frames, with stabilizer $\Gamma \cap N$, and
\[
(\Gamma \cap N)  \backslash \Gamma \simeq \mathcal{V}_{\mathrm{prim}}, \qquad (\Gamma \cap N)\gamma \mapsto F\gamma.
\]

The product $NLN^-$ is the open Bruhat cell of $G$. The following proposition establishes the homogeneous structure of the space of Lagrangian frames and its Haar measure $\dd m_{\mathcal{V}}$.

\begin{proposition}[Invariant measure in Bruhat coordinates]
\label{prop:bruhat_measure}
Under the identification $\mathcal{V} \simeq N \backslash G$, the $G$-invariant measure $m_{\mathcal{V}}$ on $\mathcal{V}$ exists. On the open Bruhat cell
\[
\pi^{-1}(\mathcal{O}) = \{A(X, I_d): A \in \GL_d(\RR), X \in \Sym_d(\RR)\},
\]
the $G$-invariant measure on $\mathcal{V}$ is, up to normalization, 
\[
\dd m_{\mathcal{V}}(A(X, I_d)) = c |\det A| \dd A \dd X,
\]
where $\dd X$ is Lebesgue measure on $\Sym_d(\RR)$ and $\dd A$ is a Lebesgue measure on $\mathrm{Mat}_d(\RR)$.
\end{proposition}

\begin{proof}
Under the identification $\mathcal{V} \simeq N \backslash G$, $F \cdot NLN^- = \pi^{-1}(\mathcal{O})$. Hence, the open Bruhat cell $NLN^-$ projects onto $\pi^{-1}(\mathcal{O})$, with 
\[
\iota(A)v_X \longmapsto A(X, I_d).
\] 
Moreover, $\Ad(\iota(A))u_X = u_{(A\trans)^{-1}XA^{-1}}$ and the determinant of the linear map on $\Sym_d(\RR)$ is 
\[
|\det(X \mapsto (A\trans)^{-1}XA^{-1})| = |\det A|^{-d-1}.
\] 
Hence, by the standard Haar measure formula for the open Bruhat cell (cf.~\cite{Kna96}),
\[
\dd g = c |\det A|^{d+1}\dd X \dd m_{\GL_d(\RR)}(A) \dd Y = c |\det A| \dd X \dd A \dd Y
\]
where $\dd X$, $\dd Y$ and $\dd A$ are Lebesgue measures on $\Sym_d(\RR)$ and $\mathrm{Mat}_d(\RR)$ respectively. 
The claimed formula follows immediately from the quotient integration formula in Bruhat coordinates; see \cite{Kna96}.
\end{proof}

Building on the $G$-invariant measure on $\mathcal{V}$, we define the Siegel transform as a systematic way to study the occurrence of short frames in symplectic lattices. See \cite{Sie45}. 

\begin{definition}[Siegel Transform for Lagrangian Frames]\label{def:Siegel_transform}
Let $\mathcal{V}_{\mathrm{prim}}$ be as in Definition \ref{def:lagframe}. For any function $f \in L^1(\mathcal{V}, m_{\mathcal{V}})$, we define the \emph{Siegel transform} of $f$ by 
\[
\widehat{f}(\Gamma g) = \sum_{v \in \mathcal{V}_{ \mathrm{prim}}} f(v g).
\]
\end{definition}

Using the identification above, we may express the Siegel transform as a sum over $(\Gamma \cap N) \backslash \Gamma$ and obtain the following mean-value formula.

\begin{proposition}[Siegel mean-value formula for $\mathcal{V}$]\label{prop:siegel_mean_value_formula}
    Let $m_{N \backslash G}$ be a $G$-invariant measure on $N \backslash G$ and $f \in L^1(\mathcal{V}, m_{N \backslash G})$. Let $\widehat{f}$ be the Siegel transform of $f$. Then there exists a $G$-invariant measure $m_{\mathcal{V}}$ such that 
    \[
    \int_{\mathcal{X}}\widehat{f}(\Lambda)\dd m_{\mathcal{X}} = \int_{\mathcal{V}}f(x)\dd m_{\mathcal{V}}.
    \]
\end{proposition}

\begin{proof}
    We first assume that $f \geq 0$. Using the identification $\mathcal{V}_{\mathrm{prim}}
    =
    F \cdot \Gamma
    \simeq
    (\Gamma \cap N) \backslash \Gamma$,
    we may write
    $$
    \widehat{f}(\Gamma g)
    =
    \sum_{\gamma \in (\Gamma \cap N) \backslash \Gamma}
    f(F\gamma g).
    $$
    Hence, by Tonelli's theorem and the usual unfolding argument,
    $$
    \begin{aligned}
    \int_{\mathcal{X}}\widehat{f}(\Lambda)\dd m_{\mathcal{X}}
    &=
    \int_{\Gamma \backslash G}
    \sum_{\gamma \in (\Gamma \cap N) \backslash \Gamma}
    f(F\gamma g)\dd m_{\mathcal{X}}(\Gamma g) \\
    &=
    c\int_{(\Gamma \cap N) \backslash G}
    f(Fg)\dd m_{(\Gamma \cap N) \backslash G}(g) \\
    &=
    c\,\mathrm{vol}((\Gamma \cap N)\backslash N)
    \int_{N \backslash G}
    f(Fg)\dd m_{N \backslash G}(Ng),
    \end{aligned}
    $$
    where $c>0$ depends only on the normalization of the invariant measures. Since $(\Gamma \cap N)\backslash N$ has finite Haar measure, the right-hand side is finite whenever $f \in L^1(\mathcal{V},m_{N \backslash G})$. Applying the same argument to $|f|$ shows that $\widehat{f} \in L^1(\mathcal{X})$ for arbitrary $f \in L^1(\mathcal{V},m_{N \backslash G})$. The formula for general $f$ then follows by decomposing $f$ into its positive and negative parts. Finally, define $m_{\mathcal{V}}$ to be the pushforward under the identification $ Ng \mapsto Fg $
    of the measure $c\,\mathrm{vol}((\Gamma \cap N)\backslash N)\,m_{N \backslash G}$. Then $m_{\mathcal{V}}$ is $G$-invariant, and
    \[
    \int_{\mathcal{X}}\widehat{f}(\Lambda)\dd m_{\mathcal{X}}
    =
    \int_{\mathcal{V}}f(x)\dd m_{\mathcal{V}}.
    \]
\end{proof}

\subsection{Span Separation Property}
We first discuss the span separation property, which provides the first
reduction principle in the frame-counting argument. For a fixed symplectic lattice, all sufficiently short primitive Lagrangian frames must span the same Lagrangian subspace. The following definition isolates the abstract form of this property.

\begin{definition}[Span separation]\label{def:span_separation}
Let $\mathscr{X}$ be a family of lattices in $\RR^n$, and for each
$\Lambda \in \mathscr{X}$ let $\mathscr{F}(\Lambda)$ be a
$\GL_d(\ZZ)$-invariant collection of primitive rank-$d$ frames in $\Lambda$.

We say that the frame system $\mathscr{F}$ satisfies
\emph{span separation} if there exists
$\varepsilon_{\mathrm{ss}}>0$ such that, for every
$\Lambda \in \mathscr{X}$ and every
$v,w \in \mathscr{F}(\Lambda)$ satisfying
\[
    \infnorm{v} < \varepsilon_{\mathrm{ss}},
    \qquad
    \infnorm{w} < \varepsilon_{\mathrm{ss}},
\]
there exists $\delta \in \GL_d(\ZZ)$ such that
\[
    w = \delta v.
\]
By primitiveness, this is equivalent to requiring that any two sufficiently
short frames in $\mathscr{F}(\Lambda)$ have the same real span.
\end{definition}

We identify $\mathcal{X} = \Gamma \backslash G$ with the space of symplectic lattices via $\Gamma g \mapsto \ZZ^{2d}g$. The following lemma gives a quantitative form of span separation for
primitive Lagrangian frames in symplectic lattices.

\begin{lemma}[Unique Short Lagrangian Subspace]\label{lem:unique_small_lag}
    Let $\Lambda = \ZZ^{2d}g$ be a symplectic lattice in $\RR^{2d}$ with $g \in \Sp_{2d}(\RR)$. Let $x \in \mathcal{V}_{\mathrm{prim}}g$ be a primitive Lagrangian frame and denote by $\pi(x)$ the Lagrangian subspace spanned by the row vectors of $x$. Assume that $\infnorm{x} \leq \varepsilon$. Then any primitive Lagrangian frame $y \in \mathcal{V}_{\mathrm{prim}} g$ with $\pi(x) \neq \pi(y)$ satisfies
    \[
    \frac{1}{2d\varepsilon} \leq \infnorm{y}.
    \]
\end{lemma}

\begin{proof}
Write $x=x_0g$ and $y=y_0g$ with
$x_0,y_0 \in \mathcal{V}_{\mathrm{prim}}$. Since $\pi(x)$ and $\pi(y)$
are distinct Lagrangian subspaces, not all cross-pairings between their
row vectors can vanish. Otherwise, $\pi(y) \subseteq \pi(x)^{\perp \omega}
=\pi(x)$, and hence $\pi(y)=\pi(x)$. Thus there exist rows $u \in \ZZ^{2d}$ of $x_0$ and
$v \in \ZZ^{2d}$ of $y_0$ such that
$\omega(ug, vg) \neq 0$. Since $g \in \Sp_{2d}(\RR)$,
\[
\omega(ug,vg)=\omega(u,v) \in \ZZ\setminus\{0\}.
\]
Therefore,
\[
1 \leq |\omega(u,v)| = |\omega(ug,vg)| \leq \|ug\|_2\|vg\|_2 \leq 2d\,\infnorm{x}\infnorm{y} \leq 2d \varepsilon\,\infnorm{y}.
\]
Hence
\[
\infnorm{y}\geq\frac{1}{2d\varepsilon},
\]
as claimed.
\end{proof}

Define $S_{\varepsilon}:=\{v \in \mathcal{V}: \infnorm{v} \leq \varepsilon\},$
and let $\chi_{S_{\varepsilon}}$ denote its indicator function.
The following corollary is a reformulation of Lemma~\ref{lem:unique_small_lag}.

\begin{corollary}[Span separation in Lagrangian frames]\label{cor:lag}
Fix $g \in G$. For any $\varepsilon < 1/\sqrt{2d}$, if $\chi_{S_{\varepsilon}}(xg)\,\chi_{S_{\varepsilon}}(yg)=1$ for some $x,y\in \mathcal{V}_{\mathrm{prim}}$, then there exists $\delta\in\GL_d(\ZZ)$ such that $y=\delta x$.
\end{corollary}

\begin{proof}
By Lemma~\ref{lem:unique_small_lag}, the frames $xg$ and $yg$ span the
same Lagrangian subspace, say $L$. Since both are primitive frames in
$\Lambda = \ZZ^{2d}g$, their row lattices are equal to the primitive rank-$d$ lattice $\Lambda \cap L$. Hence they differ by a change of basis in $\GL_d(\ZZ)$, and therefore there exists $\delta \in \GL_d(\ZZ)$ such that $y=\delta x$. Thus, primitive Lagrangian frames satisfy span separation with $\varepsilon_{\mathrm{ss}} = 1/\sqrt{2d}$.
\end{proof}

\begin{remark}\label{rmk:ss_examples}
    The mechanism behind this property is not specific to the symplectic setting. More generally, the same span separation applies to split half-dimensional arithmetic maximal isotropic frame systems whenever the underlying bilinear or Hermitian pairing detects distinct maximal isotropic subspaces.

    Let $v$ and $w$ be arithmetic maximal isotropic frames spanning
    distinct maximal isotropic subspaces. Since a maximal isotropic subspace
    $L$ satisfies $L=L^\perp$, the vanishing of all cross-pairings between
    the rows of $v$ and $w$ would force their spans to coincide. Thus some
    cross-pairing is nonzero. If the arithmetic structure makes the set of
    nonzero pairing values discrete, this yields a uniform lower bound for
    such a pairing, whereas the Cauchy--Schwarz inequality gives an upper
    bound of the form
    \[
        1 \leq |\langle v_i,w_j\rangle|
        \ll \norm{v}\norm{w}.
    \]
    Consequently, two sufficiently short arithmetic maximal isotropic frames cannot span distinct subspaces. This discussion concerns the split half-dimensional setting, in which a maximal isotropic subspace satisfies $L=L^\perp$. The same discrete-pairing argument applies, after fixing an arithmetic integral model, to the standard maximal isotropic frame systems associated with split orthogonal and unitary groups, for example
    \[
    \SO(d,d) \qquad\text{and}\qquad \SU(d,d).
    \]
    In the unitary case, one fixes a coefficient order and an integral Hermitian lattice; equivalently, one may realify the standard representation of $\SU(d,d)$, in which a unitary isotropic $d$-frame gives a rank-$2d$ real frame with complex-structure-invariant span. The relevant nonzero pairing values then form a discrete set, and the same separation argument applies. We do not develop the corresponding orthogonal or unitary frame-counting theories here; the point is only to isolate the geometric input required for such extensions.
\end{remark}

\section{Dynamical formulation of the Khintchine-type Theorem}\label{sec:dynamical_formulation}
In this section, we establish a dynamical formulation of the Khintchine-type theorem for symmetric matrices. Following \cite{Dan85}, \cite{KM99}, we reformulate the Khintchine-type theorem as a shrinking target problem in the space of symplectic lattices $\mathcal{X} = \Gamma \backslash G$.

\subsection{The Dani Correspondence}\label{subsec:dani}

We relate the Diophantine properties of $X$ to the depth of cusp excursions of the lattice $\Lambda_X = \ZZ^{2d}u_X$ in the homogeneous space $\mathcal{X}$.

\begin{definition}[The frame-cusp function]\label{def:delta}
For $\Gamma g \in \mathcal{X}$, we define the frame-cusp function $\Delta(\Gamma g)$ in terms of the supremum norm of the shortest Lagrangian frame in the lattice:
\[
\Delta(\Gamma g) := \sup \left\{ \log \frac{1}{\infnorm{v}} : v \in \mathcal{V}_{\mathrm{prim}} g \right\}.
\]
For any $R \in \RR$, we define the corresponding cusp neighborhood as $A_R := \{\Gamma g \in \mathcal{X} : \Delta(\Gamma g) \geq R \}$.
\end{definition}

Since every symplectic lattice contains primitive Lagrangian frames and only finitely many lattice frames in any bounded set, the minimum frame norm is positive and attained. To translate the approximation rate into a dynamical setting, we construct a time-dependent target depth $r(t)$ based on the approximation function $\psi$.

\begin{proposition}[Dani Correspondence]\label{prop:dani}
Let $g_t = \diag(e^t I_d, e^{-t} I_d)$ be a one-parameter diagonal subgroup in $G$. Let $\psi$ be a continuous, non-increasing function such that $\psi(q) \rightarrow 0$. Define $r(t)$ implicitly by the relation
\begin{equation}\label{eq:rt_def}
\psi(e^{t-r(t)}) = e^{-t-r(t)}.
\end{equation}
Then for every $X \in (\Sym_d(\RR)\setminus \Sym_d(\QQ)) \cap [0, 1)^{d \times d}$,
\[
X \in W(\psi) \iff \Delta(\Gamma u_X g_t^{-1}) \ge r(t) \quad \text{for arbitrarily large } t.
\]
\end{proposition}

\begin{proof}
Set
\[
T(q)=\frac{1}{2}\log\frac{q}{\psi(q)},\qquad q\geq1,
\]
and let $t_0=T(1)$. Since $\psi$ is positive, continuous, and
non-increasing with $\psi(q) \to 0$, the map
\[
T:[1,\infty)\longrightarrow[t_0,\infty)
\]
is a strictly increasing homeomorphism. For $t \geq t_0$, define
\[
q(t)=T^{-1}(t),
\qquad
r(t)=t-\log q(t).
\] 
Then 
\[
q(t)=e^{t-r(t)},
\qquad
\psi(q(t))=e^{-t-r(t)}.
\]
Thus \eqref{eq:rt_def} holds. Since only unbounded times are relevant,
we may discard the finite initial interval $[0,t_0)$. Suppose $\infnorm{QX+P} \le \psi(\infnorm{Q})$. Let $\infnorm{Q} = e^{t-r(t)}$. By \eqref{eq:rt_def}, the error is bounded by $e^{-t-r(t)}$. Applying the diagonal flow $g_t^{-1}$, the vector $v = (Q, P)u_Xg_t^{-1}$ satisfies:
\[
\infnorm{v} = \max\left( e^t \infnorm{QX+P}, e^{-t} \infnorm{Q} \right) \le \max(e^{-r(t)}, e^{-r(t)}) = e^{-r(t)}.
\]
Thus, $\Delta(\Gamma u_Xg_t^{-1}) \geq r(t)$. For the converse, using the monotonicity of $\psi$, we have infinitely many $(Q, P) \in \mathcal{V}_{\mathrm{prim}}$ satisfying
\[
\infnorm{QX + P} \leq e^{-t - r(t)} = \psi(e^{t - r(t)}) \leq \psi(\infnorm{Q}).
\]
As $t$ becomes arbitrarily large, $\infnorm{QX + P}$ becomes arbitrarily small but is not equal to $0$. If $QX + P = 0$, we have $(Q, P) = Q(I_d,  -X)$ and the full-rank condition implies $Q$ is invertible, and thus $X = - Q^{-1}P \in \Sym_d(\QQ)$. Hence, we may assume $QX + P  \neq 0$. Finally, if $\infnorm{Q}$ is bounded, then bounded error $\psi(\infnorm{Q})$ makes the integral approximating frames $(Q, P)$ finite. Therefore, after passing to a subsequence,
\[
\infnorm{Q} \longrightarrow \infty,
\]
and hence the target hits yield infinitely many distinct solutions in
the definition of $W(\psi)$.
\end{proof}

\subsection{Borel--Cantelli Property and Cusp Geometry}\label{subsec:borelcatelli_cuspgeometry}

Via the Dani correspondence, Khintchine's theorem can be reformulated as a shrinking target problem in the space $\mathcal{X}$. Following Kleinbock and Margulis, a uniformly continuous function $\Delta: \mathcal{X} \to \RR$ is said to be $\varrho$-distance-like if there exists $R_0 > 0$ such that
\[
    m_{\mathcal{X}}(\{\Gamma g \in \mathcal{X} : \Delta(\Gamma g) \geq R\}) \asymp e^{-\varrho R} \qquad \forall R \geq R_0.
\]
This distance-like property enables us to apply the dynamical
Borel--Cantelli machinery of Kleinbock and Margulis
\cite{KM99}, see also \cite{KM18}.

\begin{theorem}[Kleinbock--Margulis, Theorem 1.7~\cite{KM99}; see also \cite{KM18}]\label{thm:BC}
Let $G$ be a connected semisimple Lie group without compact factors, $\Gamma$ be an irreducible lattice in $G$, $\mathfrak{a}$ be the Lie algebra of the maximal $\RR$-split torus of $G$, and $0 \neq \mathbf{z} \in \mathfrak{a}$. Let $m_{\mathcal{X}}$ be the $G$-invariant probability measure on $\mathcal{X} = \Gamma \backslash G$. Then:

If $\Delta$ is $\varrho$-distance-like for some $\varrho > 0$, then the family of sets $\{A_{r(t)}\}_{t \in \NN}$ satisfies
\[
m_{\mathcal{X}}\left(\{x \in \mathcal{X}: x \exp(t\mathbf{z}) \in A_{r(t)} \text{ for infinitely many } t \in \NN\}\right) =  \begin{cases}
    0 & \text{if } \sum_{t=1}^\infty m_{\mathcal{X}}(A_{r(t)}) < \infty, \\
    1 & \text{if }  \sum_{t=1}^\infty m_{\mathcal{X}}(A_{r(t)}) = \infty.
\end{cases}
\]
\end{theorem}

We shall also use the passage between real and integer times. Since
\[
t-r(t)=\log q(t),
\qquad
t+r(t)=-\log\psi(q(t)),
\]
and both right-hand sides are non-decreasing, we have $|r(t)-r(s)|\leq |t-s|$. Moreover, from the definition of $\Delta$ and the diagonal form of
$g_t$,
\[
|\Delta(xg_t^{-1})-\Delta(xg_s^{-1})|
\leq |t-s|
\qquad
\forall x\in\mathcal{X}.
\]
Hence, if $\Delta(x g_t^{-1}) \geq r(t)$ and $n=\lceil t\rceil$, then
\begin{equation}\label{eq:discrete_time}
\Delta(xg_n^{-1})
\geq
r(n)-2.
\end{equation}
Thus infinitely many real-time hits at depth $r(t)$ imply infinitely
many integer-time hits at depth $r(n)-2$. Conversely, infinitely many
integer-time hits are, of course, infinitely many real-time hits.

In the subsequent subsections, we will prove that $\Delta$ is $\varrho$-distance-like with $\varrho = d^2+d$. Together with the Borel--Cantelli theorem above, this reduces the metric problem to the convergence of $\sum_{t\geq1}e^{-\varrho r(t)}$. Since the distance-like estimate is only a large-depth estimate, we
separate the shallow-target case. If $r(n)$ is bounded above along an
infinite subsequence, say $r(n_k) \leq M$, then
\[
m_{\mathcal X}(A_{r(n_k)})
\geq
m_{\mathcal X}(A_M)>0,
\]
and hence
\[
\sum_{n\geq1}m_{\mathcal X}(A_{r(n)}) = \infty
\qquad
\sum_{n\geq1}e^{-\varrho r(n)} = \infty.
\]
Otherwise $r(n) \to \infty$, and the distance-like estimate
applies for all sufficiently large $n$. Consequently,
\[
\sum_{n \geq1 }m_{\mathcal X}(A_{r(n)}) < \infty \iff \sum_{n\geq1}e^{-\varrho r(n)} < \infty.
\]
Let $s=\varrho/2$. For every $b\geq1$, Lebesgue--Stieltjes integration gives
\[
\int_{T(1)}^{T(b)}e^{-2s r(t)}\dd t
=
\int_1^b q^{s-1}\psi(q)^s\dd q
+
\frac{\psi(1)^s-(b\psi(b))^s}{2s}.
\]
Indeed, if $F(q)=(q\psi(q))^s$, then
\[
\dd T=\frac12(\dd\log q-\dd\log\psi(q)),
\qquad
\dd F=sF(\dd\log q+\dd\log\psi(q)).
\]
Since $\psi$ is non-increasing, $\dd T\geq\frac12\dd\log q$.
The displayed identity therefore implies that the time integral and
the $q$-integral are equiconvergent. Monotonicity of $\psi$ compares
the latter integral with the corresponding series, and the
Lipschitz property of $r$ compares the time integral with the
integer-time sum. Hence
\begin{equation}\label{eq:dichotomy_condition}
\sum_{n\geq1}e^{-\varrho r(n)}<\infty
\iff
\sum_{q\geq1}q^{\varsigma-1}\psi(q)^\varsigma<\infty.
\end{equation}
The same comparison yields the corresponding divergence equivalence. Thus the dynamical Borel--Cantelli criterion yields precisely the critical sum appearing in Theorem~\ref{thm:sk}.

\section{Cusp Volume Estimate}\label{sec:cusp_estimate}

In this section, we establish the distance-like property of the frame-cusp function $\Delta$. The upper cusp-volume bound follows from the first moment of the frame-counting Siegel transform. For the lower bound, the naive second moment diverges because of the infinite basis multiplicity, and we overcome this obstruction by restricting to rigid frames. The resulting second-moment estimate yields the matching lower bound and hence the sharp cusp-volume asymptotics. With the distance-like property in hand, Theorem~\ref{thm:sk} follows from the Borel--Cantelli framework of Kleinbock--Margulis~\cite{KM99}.

\begin{theorem}[Distance-like $\Delta$]\label{thm:distance_likeness}
Let $\Delta$ be the frame-cusp function defined in Definition~\ref{def:delta}. Then it is $\varrho$-distance-like, that is, $\Delta$ is uniformly continuous and there exists $R_0 > 0$ satisfying
\[
m_{\mathcal{X}}\left(\left\{\Gamma g \in \mathcal{X}: \Delta(\Gamma g) \geq R \right\}\right) \asymp e^{-\varrho R}, \qquad \varrho = d^2 + d, \;\; R \geq R_0
\]    
\end{theorem}

\subsection{Upper Bound for the Cusp Volume}\label{subsec:upper_bound}
Let $S_\varepsilon$ be the supremum norm ball of frames with norm $\leq \varepsilon$. The measure $m_{\mathcal{X}}(A_R)$ is controlled by the expected number of primitive frames in $S_{e^{-R}}$, namely $m_{\mathcal{V}}(S_{e^{-R}})$. Before applying the Siegel mean-value formula, we note that bounded frame regions have finite $m_{\mathcal{V}}$-measure. 

Let $K = G \cap \mathrm{O}(2d)$ be the maximal compact subgroup of $G$. By the $K$-invariance of $m_{\mathcal{V}}$ and the transitive action of $K$ on the compact Lagrangian Grassmannian $\mathrm{Lag}(d,2d)$, the invariant measure disintegrates over $\mathrm{Lag}(d,2d)$, up to normalization, with fiber measure
\[
|\det A|\,\dd A.
\]
For each $\mathrm{row}(X, I_d) \in \mathrm{Lag}(d, 2d)$, choose an orthonormal frame $U_X$ spanning $(X, I_d)$. On the graph chart, put
\[
H_X=(I_d+X^2)^{1/2}, \qquad U_X=H_X^{-1}(X,I_d), \qquad B=AH_X.
\]
Then $U_XU_X^{\trans}=I_d$ and $A(X,I_d)=BU_X$. Since $A=BH_X^{-1}$, the Bruhat-coordinate formula gives
\[
\dd m_{\mathcal V}(BU_X)
=
c|\det B|\,
\det(I_d+X^2)^{-(d+1)/2}\dd B\dd X.
\]
Thus the base factor is the $K$-invariant measure density on the
Lagrangian Grassmannian, while the fiber measure is
$|\det B|\dd B$. Let $B_1 = \{v \in \mathcal{V}: \norm{v}_F \leq 1\}$. Since $\norm{AU_V}_F=\norm{A}_F$, we obtain

\begin{equation}\label{eq:norm_ball_bound}
    m_{\mathcal{V}}(B_1)
\asymp
\int_{\norm{A}_F\leq 1}|\det A|\,\dd A
<\infty.
\end{equation}
By equivalence of norms, $m_{\mathcal{V}}(S_1) < \infty$ as well. As a direct corollary, we have the following upper bound for the cusp volume.

\begin{corollary}\label{cor:ubd}
    Let $A_R = \{\Gamma g \in \mathcal{X}: \Delta(\Gamma g) \geq R\}$. Then 
    \[
    m_{\mathcal{X}}(A_R) \ll e^{-\varrho R}, \qquad \varrho = d^2 + d.
    \]
\end{corollary}
\begin{proof}
    By Proposition~\ref{prop:bruhat_measure}, $\dd m_{\mathcal{V}}(A(X, I_d)) = c |\det A| \dd A \dd X$. We can directly show that
    \[
    m_{\mathcal{V}}(S_{\varepsilon}) = \varepsilon^{d^2 + d}m_{\mathcal{V}}(S_1)
    \]
    by the change of variables $A \mapsto \varepsilon A$. Define $N_{e^{-R}}(\Gamma g) = \#\{x \in \mathcal{V}_{\mathrm{prim}} g: \infnorm{x} \leq e^{-R}\}$. This is precisely the Siegel transform of $\chi_{S_{e^{-R}}}$, considered as a function on $(\mathcal{X}, m_{\mathcal{X}})$. By the Siegel mean-value theorem (Proposition~\ref{prop:siegel_mean_value_formula}),
    \[
    \int_{\mathcal{X}}N_{e^{-R}}(\Gamma g)\dd m_{\mathcal{X}}(\Gamma g) = \int_{\mathcal{V}}\chi_{S_{e^{-R}}}(x)\dd m_{\mathcal{V}}(x) = m_{\mathcal{V}}(S_{e^{-R}}) = e^{-\varrho R}m_{\mathcal{V}}(S_1).
    \]
    Since the cusp neighborhood $A_R$ is defined by the existence of a primitive Lagrangian frame, we have
    \[
    \Gamma g \in A_{R} \iff N_{e^{-R}}(\Gamma g) \geq 1.
    \]
    Therefore, by Markov's inequality,
    \[
    m_{\mathcal{X}}(A_R) = m_{\mathcal{X}}(\{\Gamma g: N_{e^{-R}}(\Gamma g) \geq 1\}) \leq \int_{\mathcal{X}}N_{e^{-R}}(\Gamma g)\dd m_{\mathcal{X}}(\Gamma g) = e^{-\varrho R}m_{\mathcal{V}}(S_1). \qedhere
    \]
\end{proof}

\subsection{Second Moment for Siegel Transform and Lower Bound for Cusp Decay}\label{subsec:lower_bound}

The second moment estimates for the Siegel transform date back to Rogers, who derived an explicit second moment formula for unimodular lattices in connection with lattice point counting problems \cite{Rog55, Rog56, Sch60}. While these results were extended to the symplectic setting by Kelmer and Yu \cite{KY21}, a key distinction in our approach is that we are concerned with counting \emph{Lagrangian frames} rather than individual lattice points. Related mean-value formulas for counting higher-rank sublattices and flags, including divergent $L^2$-phenomena, were obtained by Kim~\cite{Kim22}.

To establish the matching lower bound for the cusp decay and deduce the distance-like property of $\Delta$, we must control the second moment of our counting function. Let $N_{e^{-R}}$ denote the number of short primitive Lagrangian frames in a lattice. By the Cauchy--Schwarz inequality, we have 
\begin{equation}\label{eq:pzineq}
    m_{\mathcal{X}}\left(\left\{ \Lambda \in \mathcal{X} : N_{e^{-R}}(\Lambda) > 0 \right\}\right) \geq \frac{\mathbb{E}[N_{e^{-R}}]^2}{\mathbb{E}[N_{e^{-R}}^2]}.
\end{equation}

\medskip
\noindent\textbf{Span separation reduction.} One might naively expect that $\mathbb{E}[N_{e^{-R}}^2] \ll \mathbb{E}[N_{e^{-R}}]$ for sufficiently large $R$. However, this fails in general. Exploiting the symplectic lattice structure and appealing to Corollary~\ref{cor:lag}, we obtain, for sufficiently small $\varepsilon > 0$,
\begin{align}\label{eq:secmt}
\begin{split}
\int_{\mathcal{X}}\widehat{\chi}_{S_{\varepsilon}}^2(\Gamma g)\,\dd m_{\mathcal{X}}(\Gamma g)
&= \int_{\mathcal{X}}\sum_{x, y \in \mathcal{V}_{\mathrm{prim}}}\chi_{S_{\varepsilon}}(xg)\chi_{S_{\varepsilon}}(yg)\,\dd m_{\mathcal{X}}(\Gamma g) \\
&= \int_{\mathcal{X}}\sum_{\delta \in \GL_d(\ZZ)}\sum_{x \in \mathcal{V}_{\mathrm{prim}}}\chi_{S_{\varepsilon}}(xg)\chi_{S_{\varepsilon}}(\delta xg)\,\dd m_{\mathcal{X}}(\Gamma g) \\
&= \sum_{\delta \in \GL_d(\ZZ)}\int_{\mathcal{V}}\chi_{S_{\varepsilon}}(x)\chi_{S_{\varepsilon}}(\delta x)\,\dd m_{\mathcal{V}}(x) \\
&= \sum_{\delta \in \GL_d(\ZZ)}m_{\mathcal{V}}\left(S_{\varepsilon} \cap \delta^{-1}S_{\varepsilon}\right).
\end{split}
\end{align}

The significance of \eqref{eq:secmt} is that the span separation property in Definition~\ref{def:span_separation} reduces the entire pair correlation to a single fiber of 
\[
N \backslash G \longrightarrow P \backslash G.
\]
Thus, any two sufficiently small frames span the same Lagrangian subspace, so their position is completely described by the $\GL_d(\ZZ)$-action within that fiber.

But this single fiber reduction is not enough. The sum diverges for $d \geq 3$ due to overcounting valid Lagrangian frames. The resulting $\GL_d(\ZZ)$ Levi-multiplicity produces a heavy-tailed Siegel transform.

\begin{proposition}[Divergence of the Second Moment]\label{prop:divergence_secondmoment}
Let $d \geq 3$. For any fixed $\varepsilon > 0$, we have
\[
\int_{\mathcal{X}}\widehat{\chi}_{S_{\varepsilon}}^2(\Gamma g) \dd m_{\mathcal{X}}(\Gamma g) = \infty.
\]

\end{proposition}

\begin{proof}
    See Appendix~\ref{sec:appendix}
\end{proof}

\subsection{Rigid Frame Regularization}\label{subsec:rigid_frame_regularization}

Proposition~\ref{prop:divergence_secondmoment} shows that the naive Rogers-type second moment estimate fails in the frame-counting setting. To obtain a geometrically defined finite-overlap frame count, we restrict to a class of frames for which the $\GL_d(\ZZ)$-multiplicity is uniformly controlled. This leads to the
notion of rigid frames. Since the regularization mechanism is independent of the Lagrangian condition, we develop it in this subsection on the general $d$-frame space $\mathrm{Fr}_{d,n}$ and specialize it to Lagrangian frames only when it is applied to the second-moment problem.
\begin{definition}[Rigid frames]\label{def:rigid_frames}
For $d \leq n$, let
\[
\mathrm{Fr}_{d,n}
=
\left\{
v \in \mathrm{Mat}_{d\times n}(\RR):
\rank(v)=d
\right\},
\qquad
\mathscr W=\Or_d(\ZZ),
\]
where $\mathscr W$ is the signed permutation group. Define
\[
\Omega_{\mathrm{rig}}
=
\left\{
v \in \mathrm{Fr}_{d,n}:
\norm{\delta v}_F\geq\norm v_F
\text{ for all }\delta\in\GL_d(\ZZ),
\text{ with equality only if }\delta\in\mathscr W
\right\}.
\]
Frames in $\Omega_{\mathrm{rig}}$ are called \emph{rigid}.
\end{definition}

By definition, if both $v$ and $\delta v$ belong to $\Omega_{\mathrm{rig}}$, then $\delta\in\mathscr W$. Thus, after restricting to rigid frames, the $\GL_d(\ZZ)$-multiplicity reduces to the finite signed permutation group. From now on, we use the Frobenius norm $\|\cdot\|_F$. Since all norms are equivalent on a finite-dimensional space, this choice does not affect the volume estimates beyond multiplicative constants. Denote
\[
B_r=\{v\in\mathcal{V}:\norm v_F\leq r\},
\qquad
B_r^{d,n}
=
\{v\in\mathrm{Fr}_{d,n}:\norm v_F\leq r\}.
\] 

\begin{lemma}[Rigid frame counting]\label{lem:rigid_counting}
For every $0<\varepsilon<\varepsilon_{\mathrm{ss}}$, define
\[
N_\varepsilon'(\Gamma g)
:=
\widehat{\chi}_{B_\varepsilon\cap\Omega_{\mathrm{rig}}}(\Gamma g)
=
\sum_{x\in\mathcal{V}_{\mathrm{prim}}}
\chi_{B_\varepsilon\cap\Omega_{\mathrm{rig}}}(xg).
\]
Then, for every $g\in G$,
\[
N_\varepsilon'(\Gamma g)
\in
\{0,|\mathscr W|\},
\qquad
|\mathscr W|=2^d d!.
\]
Consequently,
\[
m_{\mathcal X}\bigl(\{N_\varepsilon'>0\}\bigr)
=
\frac{1}{|\mathscr W|}
\mathbb E[N_\varepsilon']
=
\frac{1}{|\mathscr W|}
m_{\mathcal V}
(B_\varepsilon\cap\Omega_{\mathrm{rig}}).
\]
\end{lemma}

\begin{proof}
Suppose that $xg, yg \in B_{\varepsilon} \cap \Omega_{\mathrm{rig}}$ for some $x,y\in\mathcal{V}_{\mathrm{prim}}$. Since
$B_{\varepsilon} \subset S_{\varepsilon}$ and
$\varepsilon < \varepsilon_{\mathrm{ss}}$, Corollary~\ref{cor:lag}
gives $\delta \in \GL_d(\ZZ)$ such that
\[
yg=\delta xg.
\]
Put $v = xg$. Since both $v$ and $\delta v$ are rigid, the definition
of $\Omega_{\mathrm{rig}}$ gives
\[
\norm{v}_F \leq \norm{\delta v}_F 
\leq \norm{\delta^{-1}\delta v}_F = \norm{v}_F.
\]
Thus equality holds, and hence $\delta \in \mathscr{W}$. 

Conversely, if $v \in B_{\varepsilon} \cap\Omega_{\mathrm{rig}}$, then
$\sigma v\in B_{\varepsilon} \cap \Omega_{\mathrm{rig}}$ for every
$\sigma \in \mathscr{W}$. Since $v$ has rank $d$, these $|\mathscr W|$ frames are distinct. Therefore
\[
N_{\varepsilon}'(\Gamma g)\in\{0,|\mathscr W|\}.
\]
The asserted measure identity follows from the Siegel mean-value
formula.
\end{proof}
Thus the remaining issue is geometric: we must ensure that the rigid restriction retains a positive proportion of the available small frames. The following proposition provides this regularization independently of the particular frame system.

\begin{proposition}[Rigid frame regularization]\label{prop:regularization}
Let $\mathscr{F}\subset\mathrm{Fr}_{d,n}$ be a non-empty, left $\GL_d(\RR)$-invariant subset of the $d$-frame space. Then there exists a non-empty relatively open subset
\[
\mathcal{D}
\subset
\mathscr{F}\cap\Omega_{\mathrm{rig}}\cap B_1^{d,n}
\]
such that, for every $0<\varepsilon\leq1$,
\[
\varepsilon\mathcal{D}
\subset
\mathscr{F}\cap\Omega_{\mathrm{rig}}\cap B_\varepsilon^{d,n}.
\]
\end{proposition}

To derive conditions ensuring that every Frobenius norm minimizer lies in $\mathscr W$, we introduce the following energy function.

\begin{definition}[Frobenius energy]\label{def:4.8}
For $v\in\mathrm{Fr}_{d,n}$, define the Gram matrix
\[
\theta(v)=vv\trans\in\Sym_d^+(\RR)
\]
and the associated energy function
\[
\Phi_v:\GL_d(\ZZ)\longrightarrow\RR_{>0},
\qquad
\delta\longmapsto
\tr(\delta\theta(v)\delta\trans).
\]
Writing $\delta_1\trans,\dots,\delta_d\trans$ for the rows of $\delta$, we have
\[
\Phi_v(\delta)
=
\sum_{j=1}^d
\delta_j\trans\theta(v)\delta_j.
\]
\end{definition}

The following two conditions will be used to force every minimizing $\delta$ to lie in $\mathscr W$.

\begin{definition}[$\eta$-coherence and $M$-diagonal bound]\label{def:4.9}
Fix $\eta\in(0,1)$ and $M>1$. Let $\Theta\in\Sym_d^+(\RR)$. We say that $\Theta$ is \emph{$\eta$-coherent} if
\[
|\Theta_{ij}|
<
\eta\sqrt{\Theta_{ii}\Theta_{jj}}
\qquad
\text{for all }1\leq i<j\leq d.
\]
We say that $\Theta$ has the \emph{$M$-diagonal bound} if
\[
\max_i\Theta_{ii}
<
M\min_i\Theta_{ii}.
\]
\end{definition}

If $\theta(v)$ is $\eta$-coherent, then the row vectors of the frame are uniformly close to being orthogonal. The $M$-diagonal bound further ensures that their lengths are comparable. In terms of $\Sym_d^+(\RR)$, these conditions impose a quantitative near-diagonal condition on $\theta(v)$. Such reduction principles are classical in the geometry of numbers; see, for example, \cite{Cas97}. The next proposition shows that, under a suitable choice of $\eta$ and $M$, the Frobenius-energy minimizers are exactly the elements of $\mathscr W$.

\begin{proposition}[Reduction criterion]\label{prop:cri}
Fix $M>1$ and suppose that
\[
(d-1)\eta<\frac{1}{1+dM}.
\]
Let $\Theta\in\Sym_d^+(\RR)$ be $\eta$-coherent and satisfy the $M$-diagonal bound. For any $v\in\mathrm{Fr}_{d,n}$ with $\theta(v)=\Theta$, the set of minimizers of $\Phi_v$ over $\GL_d(\ZZ)$ is $\mathscr W$. Equivalently,
\[
\left\{
v\in\mathrm{Fr}_{d,n}:
\theta(v)
\text{ is $\eta$-coherent and has the $M$-diagonal bound}
\right\}
\subset
\Omega_{\mathrm{rig}}.
\]
\end{proposition}

To prove the proposition, we first record two lemmas.

\begin{lemma}\label{lem:com}
Assume that $\Theta$ is $\eta$-coherent. Let
\[
A_\Theta
=
\diag(\Theta_{11},\dots,\Theta_{dd})
\]
be the diagonal part of $\Theta$. Then, for every $z\in\ZZ^d$,
\[
(1-(d-1)\eta)z\trans A_\Theta z
\leq
z\trans\Theta z
\leq
(1+(d-1)\eta)z\trans A_\Theta z.
\]
\end{lemma}

\begin{proof}
Write
\[
z\trans\Theta z
=
\sum_{i=1}^d\Theta_{ii}z_i^2
+
2\sum_{i<j}\Theta_{ij}z_iz_j.
\]
By $\eta$-coherence and the arithmetic-geometric mean inequality,
\[
2|\Theta_{ij}z_iz_j|
\leq
2\eta
\sqrt{\Theta_{ii}\Theta_{jj}}
|z_i||z_j|
\leq
\eta
\left(
\Theta_{ii}z_i^2+\Theta_{jj}z_j^2
\right).
\]
Summing over $i<j$ gives
\[
\left|
2\sum_{i<j}\Theta_{ij}z_iz_j
\right|
\leq
(d-1)\eta
\sum_{i=1}^d\Theta_{ii}z_i^2
=
(d-1)\eta\,
z\trans A_\Theta z,
\]
which proves the claim.
\end{proof}

\begin{lemma}\label{lem:diag}
Let $A=\diag(a_1,\dots,a_d)$ be a positive diagonal matrix. Let $v\in\mathrm{Fr}_{d,n}$ satisfy $\theta(v)=A$. Then
\[
\min_{\delta\in\GL_d(\ZZ)}
\Phi_v(\delta)
=
\sum_{j=1}^d a_j,
\]
and equality holds precisely for $\delta\in\mathscr W$.
\end{lemma}

\begin{proof}
By definition,
\[
\Phi_v(\delta)
=
\tr(\delta A\delta\trans)
=
\sum_{j=1}^d a_jr_j,
\qquad
r_j
=
\sum_{i=1}^d\delta_{ij}^2
\in\ZZ.
\]
Since $\delta$ is invertible, none of its columns is zero, and hence $r_j\geq1$ for every $j$. Therefore,
\[
\Phi_v(\delta)
\geq
\sum_{j=1}^d a_j.
\]
Equality holds precisely when $r_j=1$ for every $j$, in which case every column of $\delta$ is a signed coordinate vector. Since $\delta$ is invertible, these coordinate vectors are distinct, and hence $\delta\in\mathscr W$.
\end{proof}

\begin{proof}[Proof of Proposition~\ref{prop:cri}]
Let
\[
\Theta_{\min}
=
\min_i\Theta_{ii},
\qquad
c=(d-1)\eta,
\]
and choose $v_0\in\mathrm{Fr}_{d,n}$ such that
\[
\theta(v_0)=A_\Theta.
\]
By Lemma~\ref{lem:com}, for every $\delta\in\GL_d(\ZZ)$,
\[
(1-c)\Phi_{v_0}(\delta)
\leq
\Phi_v(\delta).
\]
For $w\in\mathscr W$,
\[
\Phi_{v_0}(w)
=
\Phi_v(w)
=
\sum_{i=1}^d\Theta_{ii}.
\]
On the other hand, if $\delta\notin\mathscr W$, then Lemma~\ref{lem:diag} gives
\[
\Phi_{v_0}(\delta)
\geq
\Theta_{\min}
+
\sum_{i=1}^d\Theta_{ii}.
\]
Hence
\[
\Phi_v(\delta)
\geq
(1-c)
\left(
\Theta_{\min}
+
\sum_{i=1}^d\Theta_{ii}
\right).
\]
It therefore suffices to have
\[
(1-c)\Theta_{\min}
\geq
c\sum_{i=1}^d\Theta_{ii}.
\]
By the $M$-diagonal bound,
\[
\sum_{i=1}^d\Theta_{ii}
\leq
d\max_i\Theta_{ii}
<
dM\Theta_{\min}.
\]
Thus it is enough that
\[
1-c>cdM,
\]
or equivalently
\[
c<\frac{1}{1+dM},
\]
which is exactly the assumed condition.
\end{proof}

\begin{proof}[Proof of Proposition~\ref{prop:regularization}]
Choose $\eta$ and $M$ so that
\[
(d-1)\eta<\frac{1}{1+dM}.
\]
Define
\[
\mathcal D
=
\left\{
v\in\mathscr F:
\begin{array}{l}
\theta(v)\text{ is $\eta$-coherent and has the $M$-diagonal bound},
\displaystyle
\min_i\theta(v)_{ii}<\frac{1}{dM}
\end{array}
\right\}.
\]
It is contained in $\Omega_{\mathrm{rig}}$ by Proposition~\ref{prop:cri}. All the defining conditions are open and invariant under multiplication by scalars in $(0,1]$, so $\mathcal D$ is relatively open in $\mathscr F$ and
\[
r\mathcal D\subset\mathcal D
\qquad
(0<r\leq1).
\]
Moreover, for every $v\in\mathcal D$,
\[
\norm v_F^2
=
\tr(\theta(v))
=
\sum_i\theta(v)_{ii}
\leq
d\max_i\theta(v)_{ii}
<
dM\min_i\theta(v)_{ii}
<
1.
\]
Thus
\[
\mathcal D
\subset
\mathscr F\cap\Omega_{\mathrm{rig}}\cap B_1^{d,n}.
\]
It remains to prove that $\mathcal D$ is non-empty. Take any $v_0\in\mathscr F$ and, for $c>0$, set
\[
x
=
c\,\theta(v_0)^{-1/2}v_0.
\]
Since $\mathscr F$ is left $\GL_d(\RR)$-invariant, we have $x\in\mathscr F$, and
\[
\theta(x)
=
c^2I_d.
\]
Taking
\[
c=\frac{1}{\sqrt{2dM}},
\]
we see that $\theta(x)$ is $\eta$-coherent and satisfies the $M$-diagonal bound, while
\[
\min_i\theta(x)_{ii}
=
\frac{1}{2dM}
<
\frac{1}{dM}.
\]
Hence $x\in\mathcal D$, proving that $\mathcal D\neq\emptyset$. Finally, since $\mathcal D$ is invariant under multiplication by $(0,1]$,
\[
\varepsilon\mathcal D
\subset
\mathscr F\cap\Omega_{\mathrm{rig}}\cap B_\varepsilon^{d,n}
\qquad
(0<\varepsilon\leq1).
\]
\end{proof}

We now apply the regularization to the Lagrangian frame space $\mathcal{V}$. This space is non-empty and left $\GL_d(\RR)$-invariant in $\mathrm{Fr}_{d,2d}$, so Proposition~\ref{prop:regularization} provides a non-empty relatively open set
\[
\mathcal D
\subset
\mathcal{V}\cap\Omega_{\mathrm{rig}}\cap B_1
\]
such that
\[
\varepsilon\mathcal D
\subset
\mathcal{V}\cap\Omega_{\mathrm{rig}}\cap B_\varepsilon
\qquad
(0<\varepsilon\leq1).
\]
The measure $m_{\mathcal{V}}$ appearing in the mean-value formula has full support and satisfies the homogeneous scaling
\[
m_{\mathcal{V}}(rE)
=
r^\varrho m_{\mathcal{V}}(E).
\]
Hence
\begin{equation}\label{eq:regularized_volume}
m_{\mathcal{V}}
(B_\varepsilon\cap\Omega_{\mathrm{rig}})
\geq
m_{\mathcal{V}}(\varepsilon\mathcal D)
=
\varepsilon^\varrho
m_{\mathcal{V}}(\mathcal D)
\gg
\varepsilon^\varrho.
\end{equation}
Thus the regularized domain retains the full cuspidal exponent.

\begin{corollary}[Lower Bound for Cuspidal Volume]\label{cor:lbd}
Let
\[
A_R
=
\{\Gamma g\in\mathcal X:\Delta(\Gamma g)\geq R\}.
\]
Then for sufficiently large $R \geq R_0$,
\[
m_{\mathcal X}(A_R)
\gg
e^{-\varrho R}.
\]
\end{corollary}

\begin{proof}
Put $\varepsilon=e^{-R}$. Since the existence of a rigid primitive Lagrangian frame in $B_\varepsilon$ implies $\Delta(\Gamma g)\geq R$,
\[
m_{\mathcal X}(A_R)
\geq
m_{\mathcal X}
\bigl(
\{\Gamma g:N_{\varepsilon}'(\Gamma g) > 0\}
\bigr).
\]
Using Lemma~\ref{lem:rigid_counting} and \eqref{eq:regularized_volume}, we obtain
\[
m_{\mathcal{X}}(A_R) \geq m_{\mathcal X}
\bigl(
\{N_{\varepsilon}'>0\}
\bigr) =
\frac{1}
{|\mathscr{W}|}m_{\mathcal{V}}(B_{\varepsilon} \cap \Omega_{\mathrm{rig}}) \gg \varepsilon^{\varrho}.
\]
\end{proof}

\begin{proof}[Proof of Theorem~\ref{thm:distance_likeness}]
By Corollaries~\ref{cor:ubd} and~\ref{cor:lbd}, we have the required two-sided exponential estimate for the cusp volume. It remains to show that $\Delta$ is uniformly continuous. Let $\opinfnorm{\cdot}$ be the operator norm induced by $\infnorm{\cdot}$ such that
\[
\opinfnorm{g} = \sup_{v \in \mathcal{V}}\frac{\infnorm{vg}}{\infnorm{v}}.
\]
Then the inequality 
\[
\infnorm{v} \opinfnorm{g^{-1}}^{-1}  \leq \infnorm{vg} \leq \infnorm{v} \opinfnorm{g}
\]
directly implies $\lvert \Delta(xg) - \Delta(x)\rvert \leq \max \left\{\left\lvert \log{\opinfnorm{g}} \right\rvert, \left\lvert \log{\opinfnorm{g^{-1}}} \right\rvert\right\}$ for any $x \in \mathcal{X}$.
    
\end{proof}

\begin{proof}[Proof of Theorem~\ref{thm:sk}]
    We follow the standard argument of \cite[\S 8.7]{KM99}. By Theorem~\ref{thm:distance_likeness}, the function $\Delta$ is $\varrho$-distance-like, while Proposition~\ref{prop:dani} identifies the Diophantine condition with the shrinking-target condition
    \[
    X \in W(\psi)\setminus \Sym_d(\QQ) \iff \Gamma u_X g_t^{-1} \in A_{r(t)} \qquad \text{for infinitely many } t.
    \]
    As a consequence of Theorem~\ref{thm:BC}, the zero-one law on $\mathcal{X}$ follows according to the convergence or divergence of 
    \[
    \sum_{t = 1}^{\infty} m_{\mathcal{X}}(A_{r(t)}),
    \]
    which, by \eqref{eq:dichotomy_condition}, is equiconvergent with 
    \[
    \sum_{q = 1}^{\infty}q^{\varsigma -1}\psi(q)^{\varsigma}.
    \]
    It remains to transfer this ambient Borel--Cantelli statement to the horospherical slice
    \[
    \{\Gamma u_X: X \in \Sym_d(\RR)\}.
    \]
    Let $P_0^-$ be a sufficiently small relatively compact identity neighborhood in 
    \[
    P^- = \left\{\begin{pmatrix}
        A & 0 \\ C & D
    \end{pmatrix} \in G \right\}.
    \]
    The multiplication map
    \[
    N \times P_0^- \rightarrow G, \qquad (u_X, p^-) \mapsto u_X p^-,
    \]
    is a local diffeomorphism, and the Haar measure on $G$ is locally equivalent to the product of the corresponding smooth measures. 

    Now compare these thickened orbits with the original horospherical orbit. Since $g_t$ contracts the unipotent radical of $P^-$ and centralizes its Levi factor, 
    \[
    K_0 : = \overline{\bigcup_{t \geq 0}g_t P_0^-g_t^{-1}}
    \]
    is compact. Moreover, 
    \[
    u_Xp^-g_t^{-1} = u_X g_t^{-1} (g_t p^- g_t^{-1}).
    \]
    Since $\Delta$ is uniformly continuous, right multiplication by elements of the fixed compact set $K_0$ changes $\Delta$ by at most a bounded amount. Then there exists $C > 0$, independent of $X$, $p^- \in P_0^-$, and $t \geq 0$, such that
    \[
    \left|\Delta(\Gamma u_Xp^- g_t^{-1}) - \Delta(\Gamma u_X g_t^{-1})\right| \leq C.
    \]
    Consequently, 
    \[
    \{\Delta(\Gamma u_X p^- g_t^{-1}) \geq r(t)  + C\} \subset \{\Delta(\Gamma u_X g_t^{-1}) \geq r(t)\} \subset \{\Delta(\Gamma u_X p^- g_t^{-1})  \geq r(t) - C\}.
    \]
    A constant shift of the target depth does not change the convergence or divergence of the corresponding target volume series. Passing to discrete time as in \eqref{eq:discrete_time}, apply Theorem~\ref{thm:BC} to $r(t)+C+2$ in the divergent case and to $r(t)-C-2$ in the convergent case. The resulting ambient limsup set is conull in the divergent case and null in the convergent case. After lifting to $G$ and restricting to the local product neighborhood $NP_0^-$, the local product decomposition and Fubini's theorem show that, for Lebesgue almost every $X$, the same conclusion holds for $\Gamma u_Xp^-$ for almost every $p^-\in P_0^-$. The preceding inclusions then transfer this conclusion to $\Gamma u_X$. This proves the assertion for Lebesgue almost every $X\in\Sym_d(\RR)$.
\end{proof}

\section{Appendix}\label{sec:appendix}

In this section, we prove the obstruction to a Dirichlet theorem in the unrestricted frame setting for general square matrices and Proposition \ref{prop:divergence_secondmoment}. 

\subsection{Frame Setting Obstruction}

\begin{lemma}[Obstruction to a Dirichlet theorem for general square matrices]\label{lem:obstruction_general_square}
Let $d \geq 2$ and $\varepsilon > 0$. For 
\[
\boldsymbol{\alpha} = (\alpha_1, \dots, \alpha_d) \in \RR^d, \qquad X_{\boldsymbol{\alpha}} = 
\begin{pmatrix}
    \alpha_1 & \alpha_2 & \cdots & \alpha_d \\ 
    0 & 1 & \cdots & 0 \\ 
    \vdots && \ddots & \vdots \\
    0 & 0 & \cdots & 1
    \end{pmatrix},
\]
the following holds for Lebesgue almost every $\boldsymbol{\alpha} \in \RR^d$: for every $C > 0$, there are only finitely many pairs
\[
(Q, P) \in \mathrm{Mat}_{d \times 2d}(\ZZ), \qquad Q \neq 0, \qquad \mathrm{rank}(Q, P) = d
\]
satisfying
\[
\infnorm{QX_{\boldsymbol{\alpha}} + P} \leq \frac{C}{\infnorm{Q}^{1/d + \varepsilon}}.
\]
In particular, no universal approximation exponent strictly larger than $1/d$ can hold under the sole non-degeneracy condition $QQ\trans + PP\trans \in \GL_d(\RR)$. 

\end{lemma}

\begin{proof}
    Write
    \[
    Q = (\mathbf{q}_1, \dots, \mathbf{q}_d), \qquad P = (\mathbf{p}_1, \dots, \mathbf{p}_d)
    \]
    in terms of their column vectors. Then 
    \[
    QX_{\boldsymbol{\alpha}} + P = ( \alpha_1 \mathbf{q}_1 + \mathbf{p}_1, \alpha_2 \mathbf{q}_1 + \mathbf{q}_2 + \mathbf{p}_2, \dots, \alpha_d \mathbf{q}_1 + \mathbf{q}_d + \mathbf{p}_d ).
    \]
    Suppose that there are infinitely many such pairs $(Q, P)$. Since $X_{\boldsymbol{\alpha}}$ is fixed, only finitely many pairs can occur with $\infnorm{Q}$ bounded. Hence, after passing to a subsequence, we may assume that $\infnorm{Q} \longrightarrow \infty$. We first claim that $\mathbf{q}_1 \neq \mathbf{0}$ for all sufficiently large terms of this sequence. Indeed, if $\mathbf{q}_1 = \mathbf{0}$, then
    \[
    QX_{\boldsymbol{\alpha}} + P = (\mathbf{p}_1, \mathbf{q}_2 + \mathbf{p}_2, \dots, \mathbf{q}_d + \mathbf{p}_d).
    \]
    Since the right-hand side of the assumed approximation inequality tends to zero, the integer vectors appearing above must vanish for all sufficiently large $\infnorm{Q}$. Thus 
    \[
    \mathbf{p}_1 = \mathbf{0}, \qquad \mathbf{p}_j = -\mathbf{q}_j, \qquad 2 \leq j \leq d.
    \]
    Consequently, 
    \[
    QQ\trans + PP\trans = 2 \sum_{j = 2}^d \mathbf{q}_j\mathbf{q}_j\trans,
    \]
    whose rank is at most $d - 1$, contradicting $QQ\trans + PP\trans \in \GL_d(\RR)$. We may therefore assume that $\mathbf{q}_1 \neq \mathbf{0}$. Choose an index $i$ such that $|(\mathbf{q}_1)_i| = \infnorm{\mathbf{q}_1}$, and put $q := |(\mathbf{q}_1)_i| \geq 1$. Taking the $i$-th coordinate in each column of $QX_{\boldsymbol{\alpha}} + P$, and absorbing the sign of $(\mathbf{q}_1)_i$ into the corresponding integers, we obtain integers $p_1, \dots, p_d$ such that
    \[
    \max_{1 \leq j \leq d}|q\alpha_j + p_j| \leq \frac{C}{\infnorm{Q}^{1/d + \varepsilon}} \leq \frac{C}{q^{1/d + \varepsilon}}.
    \]
    For Lebesgue almost every $\boldsymbol{\alpha} \in \RR^d$, $\mathbf{q}_1 \alpha_1 + \mathbf{p}_1 \neq 0$. We claim that $q \to \infty$ along a subsequence. If $q$ is bounded, then after passing to a further subsequence, we can assume that $q$ is fixed. This contradicts $\mathbf{q}_1 \alpha_1 + \mathbf{p}_1 \neq 0$ and $\infnorm{QX_{\boldsymbol{\alpha}} + P} \to 0$. Thus $q \to \infty$. Then every such pair $(Q, P)$ gives a simultaneous rational approximation to $(\alpha_1, \dots, \alpha_d)$ with exponent $1/d + \varepsilon$. This contradicts the fact that, by Khintchine--Groshev theorem, the set of very well-approximable vectors has Lebesgue measure zero.
\end{proof}

\subsection{Divergence of the Second Moment}
For $\delta \in \GL_d(\RR)$, let 
\[
\sigma_1(\delta) \geq \sigma_2(\delta) \geq \dots \geq \sigma_d(\delta) > 0
\]
be its singular values. For each Lagrangian subspace $V$, let $U_V$ be any orthonormal frame spanning $V$. Since $U_V  U_V\trans = I_d$, 
\[
\norm{A U_V}_F^2 = \tr (AU_V U_V\trans A\trans) = \norm{A}_F^2.
\]
Moreover,
\[
    \norm{AU_V}_F=\norm A_F,
    \qquad
    \norm{\delta AU_V}_F=\norm{\delta A}_F,
\]
so the conditions
\[ 
AU_V \in B_{\varepsilon} \qquad \text{and} \qquad
    \delta A U_V \in B_{\varepsilon}
\]
are independent of $V$. By the Bruhat-coordinate measure formula used in \eqref{eq:norm_ball_bound}, it suffices to study the corresponding intersection in the $\GL_d(\RR)$-fiber. In particular,
\[
m_{\mathcal{V}}
\bigl(B_{\varepsilon} \cap \delta^{-1} B_{\varepsilon}\bigr) \asymp \int_{\norm A_F \leq \varepsilon, \;\norm{\delta A}_F\leq\varepsilon}
|\det A|\,\dd A,
\]
where the implied constant is independent of $\delta$ and $\varepsilon$. 

\begin{lemma}\label{lem:overlap_lowerbound}
For every $\varepsilon > 0$, there exists a constant $c_{\varepsilon}$ depending only on $\varepsilon$, such that
\[
m_{\mathcal{V}}(B_{\varepsilon} \cap \delta^{-1}B_{\varepsilon}) \geq c_{\varepsilon}\prod_{i = 1}^d \min\{1, \sigma_i(\delta)^{-1}\}^{d+1}
\]
for every $\delta \in \GL_d(\ZZ)$.
\end{lemma}

\begin{proof}
    It suffices to estimate 
    \[
    m_{\mathcal{V}}(B_{\varepsilon} \cap \delta^{-1}B_{\varepsilon}) \asymp_{\varepsilon} \int_{\norm{A}_F < \varepsilon \; \norm{\delta A}_F < \varepsilon} |\det A| \dd A.
    \]
    Take a singular value decomposition 
    \[
    \delta = k_1 \diag(\sigma_1(\delta), \dots, \sigma_d(\delta)) k_2, \qquad k_1, k_2 \in O(d).
    \]
    Let
    \[
    D_i = \frac{\varepsilon}{2} \min\{1, \sigma_i(\delta)^{-1}\}, \qquad D_{\delta} = \diag(D_1, \dots, D_d).
    \]
    For every $B \in \mathrm{Mat}_d(\RR)$ with $\norm{B}_F < 1$, set 
    \[
    A = k_2^{-1}D_{\delta}B.
    \]
    Then we have
    \begin{align*}
        \norm{A}_F &= \norm{D_{\delta}B}_F \leq \opnorm{D_{\delta}}\norm{B}_F \leq \frac{\varepsilon}{2} \\ 
        \norm{\delta A}_F &= \norm{\diag(\sigma_1(\delta), \dots, \sigma_d(\delta)) D_{\delta}B}_F \\ &\leq \opnorm{\diag(\sigma_1(\delta), \dots, \sigma_d(\delta)) D_{\delta}} \norm{B}_F \leq \frac{\varepsilon}{2}.
    \end{align*}
    Using the change of variables
    \[
    A=k_2^{-1}D_\delta B,
    \]
    whose Jacobian is $|\det D_\delta|^d$, together with the multiplicativity of the determinant, we obtain
\begin{align*}
    \int_{\norm{A}_F < \varepsilon, \norm{\delta A}_F < \varepsilon} \vert \det A \vert \dd A &\geq \int_{\norm{B}_F < 1} \vert \det D_{\delta} \vert \; \vert \det B \vert \vert \det D_{\delta} \vert^d \dd B \\ &\gg \prod_{i = 1}^d \min\{1, \sigma_i(\delta)^{-1}\}^{d+1}.  
\end{align*}
This proves the desired inequality.
    
\end{proof}

It remains to prove that the sum $\sum_{\delta \in \GL_d(\ZZ)} \prod_{i = 1}^d \min\{1, \sigma_i(\delta)^{-1}\}^{d+1}$ diverges. It is enough to show that the summation over $\SL_d(\ZZ)$ diverges. For this, we use the lattice point counting argument in $\SL_d(\RR)$; see \cite[Theorem 1.4]{EM93}.

\begin{lemma}\label{lem:counting_cartan_region}
For $d \geq 3$, let 
\[
\alpha_i = -\log2 - \frac{i - 1}{d-1}\log2, \qquad r_d = \frac{\log2}{4(d-1)}.
\]
Then the intervals $I_i = (\alpha_i - r_d, \alpha_i + r_d)$ are disjoint. Let 
\[
\Omega = \prod_{i = 2}^{d-1} I_i
\]
be the product set of the intervals for $2 \leq i \leq d - 1$. Define a subset of $\SL_d(\RR)/\SO(d)$ by
\[
\mathcal{C}_n := \left\{ \delta \;\SO(d) \in \SL_d(\RR)/\SO(d): \sigma_1(\delta) \in [e^n, e^{n+1}), \; (\log\sigma_2(\delta), \dots, \log\sigma_{d-1}(\delta)) \in \Omega \right\}.
\]
Then for sufficiently large $n_0$, the sequence $\{\mathcal{C}_n\}_{n \geq n_0}$ is well-rounded in the sense of Eskin--McMullen~\cite{EM93}. Consequently, the number of $\SL_d(\ZZ)$-orbit points in $\mathcal{C}_n$ is comparable to
\[
m_{\SL_d(\RR)/\SO(d)}(\mathcal{C}_n) \asymp e^{2(d-1)n}.
\]
\end{lemma}

\begin{proof}
   Use the Cartan integration formula for $\SL_d(\RR)$. We have
   \[
   \dd g \asymp \prod_{1 \leq i < j \leq d}\sinh(t_i - t_j) \dd k_1 \;\dd t_1 \dots \dd t_{d-1} \; \dd k_2 
   \]
   for Haar measures $\dd k_1, \; \dd k_2$ on $\SO(d)$ with $t_i = \log \sigma_i(g)$. Since the compact $K$-factors have finite Haar measure, it suffices to estimate 
   \[
   \int_{t_1 \in [n, n+1], (t_2, \dots, t_{d-1}) \in \Omega }\prod_{1 \leq i < j \leq d}\sinh(t_i - t_j)\dd t_1 \dots \dd t_{d-1}.
   \]
   By the definition of $\Omega$, all $t_i - t_j$ are uniformly away from $0$. With $\sum_{i = 1}^d t_i = 0$, we have
   \[
   \prod_{1 \leq i < j \leq d}\sinh (t_i - t_j) \asymp \prod_{j = 2}^de^{t_1 - t_j} \prod_{2 \leq i < j \leq d}e^{t_i - t_j} \asymp e^{2(d-1)t_1}. 
   \]
   Then direct computation shows that 
   \[
   m_{\SL_d(\RR)/\SO(d)}(\mathcal{C}_n) \asymp e^{2(d-1)n}.
   \]

   Now, for any $0 < \varepsilon < r_d$, take a symmetric open identity neighborhood $O_{\varepsilon} = O_{\varepsilon}^{-1}$ in $\SL_d(\RR)$ such that
   \[
   e^{-\varepsilon}\sigma_i(g) \leq \sigma_i(ug) \leq e^{\varepsilon}\sigma_i(g), \qquad \forall i \in \{1, \dots, d\}, \quad u \in O_{\varepsilon}.
   \]
   Set
   \[
   I_j^{(\varepsilon)}
   :=
   (\alpha_j-r_d-\varepsilon,\alpha_j+r_d+\varepsilon).
   \]
Since left multiplication by an element of $O_{\varepsilon}$ changes each
$t_j=\log\sigma_j$ by at most $\varepsilon$, the set
$O_{\varepsilon} \cdot \partial \mathcal{C}_n$ is contained in the union of
coordinate boundary strips of width $O(\varepsilon)$, with all remaining
middle coordinates ranging over the enlarged intervals
$I_j^{(\varepsilon)}$. Since
\[
\alpha_i-\alpha_{i+1}
=
\frac{\log 2}{d-1}
=
4r_d,
\]
the enlarged intervals $I_i^{(\varepsilon)}$ remain pairwise disjoint for
$\varepsilon<r_d$, and hence remain uniformly away from the Weyl chamber
walls. These intervals have uniformly bounded length, so the
same Cartan integration estimate gives
\[
m_{\SL_d(\RR)/\SO(d)}
\bigl(O_\varepsilon\cdot\partial\mathcal C_n\bigr)
\ll
\varepsilon e^{2(d-1)n}
\asymp
\varepsilon\,
m_{\SL_d(\RR)/\SO(d)}(\mathcal C_n),
\]
uniformly for all sufficiently large $n$. Define the outer and inner perturbations
   \[
   \mathcal{C}_n^{+}
   :=
   \bigcup_{u \in O_{\varepsilon}} u\mathcal{C}_n,
   \qquad
   \mathcal{C}_n^{-}
   :=
   \bigcap_{u \in O_{\varepsilon}} u\mathcal{C}_n.
   \]
   By the singular-value estimate above, every point of
   $\mathcal{C}_n^{+} \setminus \mathcal{C}_n^{-}$ lies in the
   $O(\varepsilon)$-neighborhood of the coordinate boundary estimated above. Hence,
   \[
   m_{\SL_d(\RR)/\SO(d)}
   \left(\mathcal{C}_n^{+} \setminus \mathcal{C}_n^{-}\right)
   \ll
   \varepsilon \,
   m_{\SL_d(\RR)/\SO(d)}(\mathcal{C}_n)
   \]
   uniformly for all sufficiently large $n$. Therefore,
   \[
   m_{\SL_d(\RR)/\SO(d)}(\mathcal{C}_n^{+})
   \leq
   (1+O(\varepsilon))
   m_{\SL_d(\RR)/\SO(d)}(\mathcal{C}_n)
   \]
   and
   \[
   m_{\SL_d(\RR)/\SO(d)}(\mathcal{C}_n^{-})
   \geq
   (1-O(\varepsilon))
   m_{\SL_d(\RR)/\SO(d)}(\mathcal{C}_n),
   \]
   which verifies the well-roundedness condition of Eskin--McMullen~\cite{EM93}. Applying \cite[Theorem~1.4]{EM93}, and using
   $\#(\SL_d(\ZZ) \cap \SO(d)) < \infty$, we conclude that
   \[
   \# \left\{ \delta \in \SL_d(\ZZ): \delta \;\SO(d) \in \mathcal{C}_n \right\}
   \asymp
   m_{\SL_d(\RR)/\SO(d)}(\mathcal{C}_n)
   \asymp e^{2(d-1)n}.
   \]
\end{proof}

\begin{proof}[Proof of Proposition \ref{prop:divergence_secondmoment}]
It suffices to consider $0 < \varepsilon < 1/\sqrt{2d}$. For arbitrary $\varepsilon > 0$, take 
\[
\varepsilon_0 < \min\{\varepsilon, \frac{1}{\sqrt{2d}}\}.
\]
Since $\chi_{S_{\varepsilon_0}} \leq \chi_{S_{\varepsilon}}$, the divergence for $\varepsilon_0$ implies the divergence for $\varepsilon$. Recall the second moment identity \eqref{eq:secmt} for $\varepsilon < 1/\sqrt{2d}$,
\[
\int_{\mathcal{X}}\widehat{\chi}_{S_{\varepsilon}}(\Gamma g)^2 \dd m_{\mathcal{X}}(\Gamma g) = \sum_{\delta \in \GL_d(\ZZ)}m_{\mathcal{V}}(S_{\varepsilon} \cap \delta^{-1} S_{\varepsilon}).
\]
By the norm compatibility and $\SL_d(\ZZ) \subset \GL_d(\ZZ)$, it is enough to show that 
\[
\sum_{\delta \in \SL_d(\ZZ)}m_{\mathcal{V}}(B_{\varepsilon} \cap \delta^{-1}B_{\varepsilon}) = \infty.
\]
By Lemma \ref{lem:overlap_lowerbound}, it suffices to estimate 
\[
\sum_{\delta \in \SL_d(\ZZ)}\prod_{i = 1}^{d} \min\{1, \sigma_i(\delta)^{-1}\}^{d+1}.
\]
Restricting the sum to the regions $\mathcal{C}_n$ and using Lemma~\ref{lem:counting_cartan_region}, we obtain
\[
\sum_{\delta \in \SL_d(\ZZ)}\prod_{i = 1}^d \min\{1, \sigma_i(\delta)^{-1}\}^{d+1} \geq \sum_{n \geq n_0 }\sum_{\substack{\delta \SO(d) \in \mathcal{C}_n \\ \delta \in \SL_d(\ZZ)}}\sigma_1(\delta)^{-d-1} \gg \sum_{n \geq n_0}e^{2(d-1)n}e^{-(d+1)n} = \sum_{n \geq n_0}e^{(d-3)n}.
\]
Therefore, if $d \geq 3$, the series diverges. 
    
\end{proof}

\end{document}